\documentclass[reqno]{amsart}
\usepackage{amscd}

\def\beq{\begin{equation}}
\def\eeq{\end{equation}}
\def\ba{\begin{array}}
\def\ea{\end{array}}
\def\S{\mathbb S}
\def\R{\mathbb R}

\def\nn{\nonumber}

\def\la{\langle}
\def\ra{\rangle}

\def \ds{\displaystyle}
\def \vs{\vspace*{0.1cm}}

\def\a{\alpha}

\def\bee{\begin{eqnarray}}
\def\beee{\begin{eqnarray*}}
\def\eee{\end{eqnarray}}
\def\eeee{\end{eqnarray*}}
\def\nn{\nonumber}
\def\ba{\begin{array}}
\def\ea{\end{array}}

\def\slashii#1{\setbox0=\hbox{$#1$}             
   \dimen0=\wd0                                 
   \setbox1=\hbox{\sl/} \dimen1=\wd1            
   \ifdim\dimen0>\dimen1                        
      \rlap{\hbox to \dimen0{\hfil\sl/\hfil}}   
      #1                                        
   \else                                        
      \rlap{\hbox to \dimen1{\hfil$#1$\hfil}}   
      \hbox{\sl/}                               
   \fi}                                         %
\def\slashiii#1{\setbox0=\hbox{$#1$}#1\hskip-\wd0\hbox to\wd0{\hss\sl/\/\hss}}
\newcommand{\A}{{\mathcal A}}

\newcommand{\C}{{\mathbf C}}

\newtheorem{thm}{Theorem}[section]
\newtheorem{lm}[thm]{Lemma}
\newtheorem{prop}[thm]{Proposition}

\theoremstyle{definition}
\newtheorem{rem}[thm]{Remark}

\newtheorem{df}[thm]{Definition}

\theoremstyle{remark}

\begin{document}
\pagestyle{plain}
\date{\today}

\title{Vanishing Pohozaev constant and removability of singularities}

\thanks{}
\author{J\"urgen Jost, Chunqin Zhou, Miaomiao Zhu}

\address{Max Planck Institute for Mathematics in the Sciences, Inselstr. 22, D-04013 Leipzig, Germany}
\email{jjost@mis.mpg.de}

\address{Department of Mathematics and MOE-LSC, Shanghai Jiaotong University, 200240, Shanghai, China}
\email{cqzhou@sjtu.edu.cn}

\address{School of Mathematical Sciences,  Shanghai Jiao Tong University,  Dongchuan Road 800,  Shanghai, 200240, China}
\email{mizhu@sjtu.edu.cn}

\begin{abstract}{Conformal invariance of two-dimensional variational problems is a condition known to enable a blow-up analysis of solutions and to deduce the removability of singularities. In this paper, we identify another condition that is not only sufficient, but also necessary for such a removability of singularities. This is the validity of the Pohozaev identity.
In situations where such an identity fails to hold, we introduce a new quantity, called the {\it  Pohozaev constant}, which on one hand measures the extent to which
the Pohozaev identity fails and on the other hand provides a characterization of the singular behavior of a solution at an isolated singularity.
We apply this to the blow-up analysis for super-Liouville type equations on Riemann surfaces with conical singularities, because in the presence of such singularities,
 conformal invariance no longer holds and a local singularity is in general non-removable unless the Pohozaev constant is vanishing. }
\end{abstract}

\maketitle



\

\section{Introduction}\label{intro}
Many variational problems of profound interest in geometry and physics are borderline cases of the Palais-Smale condition, and standard theory does not apply to deduce the existence and to control the behavior of solutions. One needs additional ingredients and tools. For two-dimensional problems, like harmonic maps from Riemann surfaces (or in physics, the nonlinear sigma model), minimal and prescribed mean curvature surfaces in Riemannian manifolds, pseudoholomorphic curves, Liouville type problems as occurring for instance in prescribing the Gauss curvature of a surface, Ginzburg-Landau and Toda type problems, and as inspired by quantum field theory and super string theory, Dirac-harmonic maps and super-Liouville equations, etc., it turned out that conformal invariance is a key property that enables a successful analysis. The fundamental technical aspect of all such problems is the existence of bubbles, that is, the concentration of solutions at isolated points. Since the
 fundamental work of Sacks-Uhlenbeck \cite{SU} and Wente \cite{W}, we know that even when such a bubble splits off, the remaining solution is smooth, that is, can be extended through the point where the bubble singularity had been developing. This is called blow-up analysis, and it depends on a precise characterization of the bubble type solutions. In other words, conformal invariance  is a sufficient condition for such a blow-up analysis. In technical terms, conformal invariance produces a holomorphic quadratic differential. For harmonic map type problems, it is well known that finiteness of the energy functional in question implies that that differential is in $L^1$.
This then yields important estimates. For (super-)Liouville equations,
the energy functional and the holomorphic quadratic differential are defined in a different way.
Finiteness of the energy is not sufficient to get the $L^1$ bound of that differential and
hence this is an extra assumption leading to the removability of local singularities (Prop 2.6, \cite{JWZZ1}).

 It turns out, however, that some important problems in the class mentioned no longer satisfy conformal invariance. An example that we shall investigate in this paper are (super-)Liouville equations  on surfaces with conical singularities. Another example, which we shall treat in a subsequent paper, is the super-Toda system. Also, some inhomogeneous lower order terms in a problem can destroy conformal invariance.

Thus, in order to both understand the scope of the blow-up analysis in general and to handle some concrete two-dimensional geometric variational problems, we have searched for a condition that is not only sufficient, but also necessary for the blow-up analysis. The condition that we have identified is the Pohozaev identity. This condition is already known to play a crucial role in geometric analysis (see for instance \cite{St}),
but what is new here is that we can show that this identity by itself suffices for the blow-up analysis.
In fact, there are situations where this identity fails to hold. In order to handle these more complicated cases,
we introduce a new quantity that is associated to a solution, called the {\it  Pohozaev constant}. By definition, this quantity measures the extent to which
the Pohozaev identity fails. In other words, that identity holds iff the Pohozaev constant vanishes. On the other hand, it turns out
 that this quantity also provides a characterization of the singular behavior of a solution at an isolated singularity.
As already mentioned, we demonstrate the scope of this strategy at a rather difficult and subtle example, the (super-)Liouville equation  on surfaces with conical singularities. We hope that the general scheme will become clear from our treatment of this particular example.

Thus, in order to get more concrete, we now introduce that example.
The classical Liouville functional  for a real-valued function $u$ on a smooth Riemann surface $M$ with conformal metric $g$ is
\begin{equation*}
 E\left( u \right) =\int_{M}\{\frac 12 \left|
\nabla u\right| ^2+K_gu -e^{2u}\}dv,
\end{equation*}
where  $K_g$ is the Gaussian curvature of $M$.
 The Euler-Lagrange equation for $E(u)$ is the Liouville equation
\begin{equation*}
-\Delta_g u = \ds\vs 2e^{2u}  -K_g.
\end{equation*}
Liouville \cite{L} studied this equation in the plane, that is, for $K_g=0$.
The Liouville equation comes up in many problems of complex analysis
and differential geometry of Riemann surfaces, for instance
the prescribing curvature problem. The interplay between the
geometric and analytic aspects makes the Liouville equation
mathematically very interesting.

It also occurs naturally in string theory as
discovered by Polyakov \cite{Po2}, from the gauge anomaly in
quantizing the string action. There then also is a natural
supersymmetric version of the Liouville functional and equation,
coupling the bosonic scalar field to a fermionic spinor field. It
turns out, however, that we also obtain a very interesting
mathematical structure if we consider ordinary instead of
fermionic (Grassmann valued) spinor fields. Therefore, in \cite{JWZ}, we have introduced the {\bf super-Liouville
  functional}, a conformally
invariant functional that couples a real-valued function and  a spinor
$\psi$ on a closed smooth  Riemannian surface $M$ with conformal metric $g$ and a
spin structure,
\begin{equation*}
E\left( u,\psi \right) =\int_{M}\{\frac 12 \left| \nabla u\right|
^2+K_gu+\left\langle (\slashiii{D}+e^u)\psi ,\psi \right\rangle
-e^{2u}\}dv.
\end{equation*}
The Euler-Lagrange system for $E(u,\psi )$ is
\begin{equation*}
\left\{
\begin{array}{rcl}
-\Delta_g u &=& \ds\vs 2e^{2u}-e^u\left\langle \psi ,\psi
\right\rangle  -K_g\qquad
\\
\slashiii{D}_g\psi &=&\ds  -e^u\psi
\end{array} \text {in } M.
\right.
\end{equation*}
The analysis  of classical Liouville type equations was developed in  \cite{BM, LS, Li, BCLT, JLW}, and the corresponding analysis  for super-Liouville equations  in
\cite{JWZ, JWZZ1, JWZZ2, JZZ, JZZ1}. In particular,   the complete blow-up theory for sequences of solutions was established,
including the energy identity for the spinor part, the blow-up value at blow-up points and the profile for a sequence of solutions at the blow-up points. For results by physicists about super-Liouville equations, we refer to \cite{Pr, ARS, FH}.

\

In this paper, as an application and a test of our general scheme, we shall study super-Liouville equations on surfaces with conical singularities and establish the geometric and analytic properties for this system.
For this purpose, let us first recall the definition of surfaces with conical singularities, following \cite{T1}. A conformal
metric $g$ on a Riemannian surface $M$ without
boundary has a conical singularity of order $\alpha$ (a real number
with $\alpha>-1$) at a point $p\in M$ if in some neighborhood of $p$
$$g=e^{2u}|z-z(p)|^{2\alpha}|dz|^2$$
where $z$ is a coordinate of $M$ defined in this neighborhood
and $u$ is smooth away from $p$ and continuous at $p$. The point
$p$ is then said to be a conical singularity of {\it angle}
$\theta=2\pi(1+\alpha)$. For example, a (somewhat idealized) American football has two singularities of equal angle, while
a teardrop has only one singularity. Both these examples correspond
to the case $-1<\alpha <0$; in case $\alpha >0$, the angle is larger
than $2\pi$, leading to a different geometric picture. Such
singularities also appear in orbifolds and branched coverings. They
can also describe the ends of complete Riemann surfaces with
finite total curvature. If $(M, g)$ has conical
singularities of order $\alpha_1, \alpha_2, \cdots, \alpha_m$ at
$q_1, q_2, \cdots, q_m$, then $g$ is said to represent the
divisor $\A= \Sigma^m_{j=1}\alpha_j q_j$. Importantly, the presence of such conical singularities destroys conformal invariance, because the conical points are different from the regular ones.

\

Let $(M,\A, g)$ be a compact Riemann surface (without boundary) with conical singularities of divisor $\A $ and with a spin structure.
Associated to $g$, one can define the gradient $\nabla$ and the Laplacian operator $\Delta $ in the usual way.
We consider the {\bf super-Liouville functional} on $M$, a conformally
invariant functional that couples a real-valued function $u$ and a spinor
$\psi$ on $M$
\begin{equation*}
E\left( u,\psi \right) =\int_{M}\{\frac 12 \left| \nabla u\right|
^2+K_gu+\left\langle (\slashiii{D}+e^u)\psi ,\psi \right\rangle
-e^{2u}\}dv_g.
\end{equation*}
The Euler-Lagrange system for $E(u,\psi )$ is
\begin{equation}
\left\{
\begin{array}{rcl}
-\Delta_g u &=& \ds\vs 2e^{2u}-e^u\left\langle \psi ,\psi
\right\rangle  -K_g\qquad
\\
\slashiii{D}_g\psi &=&\ds  -e^u\psi
\end{array} \text {in } M\backslash \{q_1,q_2,\cdots, q_m\}.
\right.   \label{eq-2}
\end{equation}
When $\psi$ vanishes, we obtain the classical Liouville equation, or the prescribing curvature equation on M with conical singularites
(see \cite{T1, CL1}). In \cite{bt, BT1, B, Ta, BCLT, BaMo}, the blow-up theory of the following Liouville type equations with singular data was systematically studied:
$$
-\Delta_g u=\lambda \frac{Ke^u}{\int_M K e^{u}dg}-4\pi(\Sigma_{j=1}^{m}\alpha_j\delta_{q_j}- f),
$$ where $(M,g)$ is a smooth surface and the singular data appear in equation. In this paper, we aim to provide an analytic foundation for the system (\ref{eq-2}).

The local super-Liouville type system (which is deduced in the Section \ref{local}) we shall study  is the following:
\begin{equation}\label{Eq-LL}
\left\{
\begin{array}{rcl}
-\Delta u(x) &=& 2V^2(x)|x|^{2\alpha}e^{2u(x)}-V(x)|x|^{\alpha}e^{u(x)}|\Psi|^2  \quad\\
\slashiii{D}\Psi &=& -V(x)|x|^{\alpha}e^{u(x)}\Psi
\end{array} \text { in } B_{r}(0),
\right.
\end{equation}
Here $\alpha \geq 0$, $V(x)$ is a $C^{1,\beta}$ function satisfying $0< a\leq V(x)\leq b$ and $B_r=B_r(0)$ is a disc in $\R^2$.
We also assume that $(u,\Psi)$ satisfy the following energy condition:
\begin{equation}\label{Co-LL}
\int_{B_r(0)}|x|^{2\alpha}e^{2u}+|\Psi|^4dx<+\infty.
\end{equation}

Our first result is the following Brezis-Merle type concentration compactness:

\begin{thm}\label{mainthm}
Let $(u_n,\Psi _n)$ be a sequence of solutions satisfying
\begin{equation} \label{Eq-Sn}
\left\{
\begin{array}{rcl}
-\Delta u_n(x) &=& 2V^2(x)|x|^{2\alpha_n}e^{2u_n(x)}-V(x)|x|^{\alpha_n}e^{u_n(x)}|\Psi_n|^2  \quad\\
\slashiii{D}\Psi_n &=& -V(x)|x|^{\alpha_n}e^{u_n(x)}\Psi_n \end{array} \text {in } B_{r},
\right.
\end{equation}
with the energy condition
\begin{equation}\label{Eq-Cn}
\int_{B_r}|x|^{2\alpha_n}e^{2u_n}dx<C,\text{ and }\int_{B_r}\left| \Psi _n\right|
^4dx<C.
\end{equation}
for some constant $C>0$. Assume that
\begin{itemize}
\item[i)]  $\alpha_n \in \R ^+, \alpha_n \rightarrow \alpha$ with $\alpha\geq 0$,
\item[ii)] $V \in C^{1,\beta}(B_r), 0< a\leq V(x)\leq b < + \infty$.
\end{itemize}

Define
\[\begin{array}{rcl}
\Sigma _1&=&\left\{ x\in B_r,\text{ there is a sequence
}y_n\rightarrow x\text{ such that }u_n(y_n)\rightarrow +\infty
\right\}\\
\Sigma _2&=&\left\{ x\in B_r,\text{ there is a sequence
}y_n\rightarrow x\text{ such that }\left| \Psi _n(y_n)\right|
\rightarrow +\infty \right\}
.\end{array}
\]
Then, we have $\Sigma_2\subset\Sigma_1$. Moreover,
$(u_n,\Psi _n)$ admits a subsequence,
 still denoted by $(
u_n,\Psi _n),$ satisfying

\begin{enumerate}
\item[a)] $\Psi _n$ is bounded in $%
L_{loc}^{\infty} (B_r\backslash \Sigma _2)$ .

\item[b)]  For $u_n$, one of the following alternatives holds:
\begin{enumerate}
\item[i)]  $u_n$ is bounded in $L^{\infty}_{loc} (B_r).$

\item[ii)]  $u_n$ $\rightarrow -\infty $ uniformly on compact subsets of $B_r$.

\item[iii)]  $\Sigma _1$ is finite, nonempty and either
\begin{equation*}
u_n\text{ is bounded in }L_{loc}^\infty (B_r\backslash \Sigma _1)
\end{equation*}
or
\begin{equation*}
u_n\rightarrow -\infty \text{ uniformly on compact subsets of
}B_r\backslash \Sigma _1.
\end{equation*}
\end{enumerate}
\end{enumerate}
\end{thm}

\

The proof of this concentration result does not yet need the Pohozaev identity. But we shall then proceed to the subtler aspects of the blow-up analysis, and for that, the Pohozaev identity will play a crucial role. We shall first show that global singularities can be removed, that is, an entire solution on the plane can be conformally extended to the sphere. In the subsequent analysis, we shall show  that in the blow-up process, no energy will be lost, neither in the Liouville part $u_n$ nor in the spinor part $\Psi_n$.
The technically longest part of our scheme (see Section \ref{ener}) consists in exploring the blow-up behavior of (\ref{Eq-Sn}) and (\ref{Eq-Cn}) at each blow-up point, to show that the energy identity holds for the spinor parts $\Psi_n$.

\begin{thm}\label{engy-indt}
Notations and assumptions as in Theorem \ref{mainthm}. Then there are finitely many bubbling solutions of (\ref{Eq-LL}) and (\ref{Co-LL}) on $\R^2$
 with $\alpha\geq 0$ and  $V\equiv const $: $(u^{i,k},\Psi^{i,k})$, $i=1,2,\cdots , l;
k=1,2,\cdots, L_i$, all of which can be conformally extended to $S^2$, such that, after selection of a subsequence, $\Psi_n$
converges in $C_{loc}^{2}$ to some $\Psi$ on $B_r(0)\backslash \Sigma_1$
and the following energy identity holds: \begin{equation*}
\lim_{n\rightarrow
\infty}\int_{B_r(0)}|\Psi_n|^4dv=\int_{B_r(0)}|\Psi|^4dv+\sum_{i=1}^{l}
\sum_{k=1}^{L_i}\int_{S^2}|\Psi^{i,k}|^4dv. \end{equation*}
\end{thm}

\

The essential step in the  proof of Theorem \ref{engy-indt}  is the removability of a local singularity for solutions of (\ref{Eq-LL}) and (\ref{Co-LL}) defined on a punctured disc (see Section \ref{pohozaev}).

In order to see the scope of our result, we point out that,
in general, a local singularity of $(u,\Psi)$ is not removable. For
example, when $\alpha=0$, if we set
\bee\label{example}
u(x)=\log\frac{(2+2\beta)|x|^{\beta}}{1+2|x|^{2+2\beta}},
\eee
then $u$ is a solution of
$$
-\Delta u=2e^{2u},\quad \text{ in } \R^2\backslash\{0\}
$$ where $\beta >-1$. Therefore $(u,0)$ is a solution of (\ref{Eq-LL}) with $\alpha=0$ and with finite energy in
$\R^2\backslash\{0\}$. It is clear that $x=0$ is a local singularity
which is not removable when $\beta\not = 0$.

So, one needs to find some sufficient condition to remove the local singularity. In \cite{JWZZ1},  the authors considered the following simpler case of $\alpha=0$ and $V(x) \equiv1$:
 \begin{equation*}
\left\{
\begin{array}{rcl}
-\Delta u &=& \ds\vs 2 e^{2u}-e^u\left\langle \psi ,\psi
\right\rangle,\\
\slashiii{D}\psi &=&\ds  - e^u\psi.
\end{array}\qquad \text { in } B_{r_0}\setminus \{0\}
\right.
\end{equation*}
In this case, they defined the following quadratic differential
\[
T(z)dz^2=\{(\partial_z u)^2-\partial^{2}_{z}u+\frac 14\langle
\psi,dz\cdot\partial_{\bar z}\psi\rangle +\frac 14\langle
d\bar{z}\cdot\partial_z\psi,\psi\rangle  \}dz^2,
\]
and showed that it is  holomorphic in $B_{r_0}\backslash\{0\}$. Then one observes that $\int_{B_r(0)}|T(z)|dz=+\infty$ for
$(u,0)$ in the above example \eqref{example}. So, in \cite{JWZZ1}, the authors proposed the assumption that $\int_{B_r(0)}|T(z)|dz\leq C$ and showed that this is a sufficient condition
 for the removability of a local singularity. However, in the more general case considered in this paper, namely, when $\alpha>0$ or the coefficient function $V(x)$  is nonconstant,
then we do not have such a holomorphic quadratic differential and the argument in \cite{JWZZ1} does not work. Therefore we need
to develop a new method.

To describe our new method, as it applies to the super-Liouville system, let $(u, \Psi)$ be a solution of (\ref{Eq-LL}) and (\ref{Co-LL}) defined on a punctured disc.
 We define a quantity $C(u, \Psi) \in \R$,
called {\it the Pohozaev constant} associated to  $(u, \Psi)$ (see Definition \ref{poho-const}). We shall show that there is a constant $\gamma < 2\pi (1+\a)$ such that
$$u(x)=- \frac{\gamma}{2\pi}{\rm log} |x| + h, \quad {\rm near }\ 0,$$
where $h$ is bounded near $0$. Moreover, we show that $C(u, \Psi) $ and $\gamma$ satisfy the following relation:
$$C(u, \Psi)=\frac{\gamma^2}{4\pi}.$$
In particular,  we can prove that the local singularity for $(u, \Psi)$ is removable if and only if the associated Pohozaev constant $C(u, \Psi)=0$, which is equivalent to the fact
that the Pohozaev type identity for $(u, \Psi)$ holds  (see Theorem \ref{thm-sigu-move1}).

Looking back to the example \eqref{example} illustrated above, it is easy to see that the Pohozaev constant $C(u,0)=\pi \beta^2 \neq 0$ when $\beta \neq 0$.

\

Moreover, applying our new method to the removability of a local singularity, we shall see in Section \ref{blow} that the energy identity for the spinor will enable us to derive

\begin{thm}\label{mainthm1}
Notations and assumptions as in Theorem \ref{mainthm}. Assume that the blow-up set $\Sigma_1\neq \emptyset$. Then
$$u_n\rightarrow -\infty  \quad \text{ uniformly on compact subsets of }  B_r(0)\setminus \Sigma_1.$$
Furthermore,
$$2V(x)|x|^{2\alpha_n}e^{2u_n}-V(x)|x|^{\alpha_n}e^{u_n}|\Psi_n|^2\rightharpoonup\sum_{x_i\in \Sigma_1}\beta_i\delta_i
$$
in the sense of distributions, and $\beta_i\geq 4\pi$ for $x_i\in \Sigma_1\cap B_r(0)\setminus \{0\}$ and  $\beta_i\geq 4\pi(1+\alpha)$ for $x_i\in \Sigma_1\cap  \{0\}$.
\end{thm}

To investigate further the blow-up behavior of a sequence of solutions of (\ref{Eq-Sn}) and (\ref{Eq-Cn}), let us
define the blow-up value at a blow-up point $p \in \Sigma_1$ as follows:
\bee \label{mp}
m(p)=\lim_{\rho\rightarrow 0}\lim_{n\rightarrow \infty}\int_{B_\rho(p)}(2V^2(x)|x|^{2\alpha_n}e^{2u_n}-V(x)|x|^{\alpha_n}e^{u_n}|\Psi_n|^2)dx.
\eee

In Section \ref{value}, we shall then obtain

\begin{thm}\label{BV}
Notations and assumptions as in Theorem \ref{mainthm}. Assume that the blow-up set $\Sigma_1\neq \emptyset$. Let $p\in \Sigma_1$ and assume that $p$ is the only blow-up point in $\bar{B}_{\rho_0}(p)$ for some small $\rho_0>0$. If
\begin{equation}\label{bc}
 \max_{\partial B_{\rho_0}(p)}u_n-\min_{\partial B_{\rho_0}(p)}u_n\leq C,\\
\end{equation}
then the blow-up value $m(p)=4\pi$ when $p\neq 0$ and $m(p)=4\pi(1+\alpha)$ when $p=0$.
\end{thm}

\

For the global super-Liouville equations, if we let $(M, \A, g)$ be a compact Riemann surface with conical singularities represented by the divisor $\A=\Sigma_{j=1}^{m}\alpha_j q_j$, $\alpha_j>0$ and with a spin structure. Writing $g=e^{2\phi}g_0$, where $g_0$ is a smooth metric on $M$,  in Section \ref{global}, we can deduce from the results for the local super-Liouville equations:

\begin{thm} \label{thmsin} Let $(u_n,\psi_n)$ be a sequence of solutions of (\ref{eq-2}) with energy conditions:
$$
\int_{M}e^{2u_n}dg<C,~~~~~~\int_{M}|\psi_n|^4dg<C.
$$
Define
$$
\Sigma _1=\left\{ x\in M,\text{ there is a sequence
}y_n\rightarrow x\text{ such that }u_n(y_n)\rightarrow +\infty
\right\}.
$$
Then there exists $G\in W^{1,q}(M,g_0)\cap C_{loc}^2(M\backslash \Sigma_1)$
with $\int_{M}Gdg_0=0$ for $1<q<2$ such that
$$
u_n+\phi-\frac 1{|M|}\int_{M}(u_n+\phi)dg_0\rightarrow G
$$
in $C_{loc}^2(M\backslash \Sigma_1)$ and weakly in $W^{1,q}(M,g_0)$.
Moreover, in $\Sigma_1=\{p_1,p_2,\cdots, p_l\}$, then for $R>0$
small such that $B_R(p_k)\cap \Sigma_1 \cap \{q_1,...,q_m\}=\{p_k\}$, $k=1,2,\cdots, l$,
we have
\begin{equation} \label{greenfct}
G(x)=\left\{
\begin{array}{rcl}
- \frac 1{2\pi}m(p_k)\log d(x,p_k)+g(x),  \quad &{\text if}&  p_k \neq q_1,...,q_m  \\
- (\frac 1{2\pi}m(p_k)- \alpha_j )\log d(x,p_k)+g(x),  \quad    &{\text if}&   p_k = q_j, j=1,...,m
 \end{array}
\right.
\end{equation}
for $x\in B_R(p_k)\backslash \{p_k\}$ with $g\in C^2(B_R(p_k))$, where $d(x,p_k)$ denotes the Riemannian distance between $x$ and $p_k$ with respect to $g_0$ and
$$m(p_k)=\lim_{R\rightarrow 0}\lim_{n\rightarrow \infty}\int_{B_R(p_k)}(2e^{2(u_n+\phi)}-e^{u_n+\phi}|e^{\frac{\phi}{2}}\psi_n|^2-K_{g_0})dg_0,
$$

\end{thm}

\
 It is clear from the above theorem that
\begin{equation*}
 \max_{\partial B_{\rho_0}(p)}u_n-\min_{\partial B_{\rho_0}(p)}u_n\leq C,\\
\end{equation*}if $p\in \Sigma_1$ and $p$ is the only blow-up point in $\bar{B}_{\rho_0}(p)$ for some small $\rho_0>0$.
Then we get the blow-up value $m(p)=4\pi$ when $p$ is not a conical singularity of $M$ and $m(p)=4\pi(1+\alpha)$ when $p$ is a conical singularity  of $M$ with order $\alpha$.

On the other hand, on the surface $(M,\A,g)$ with the divisor $\A=\Sigma_{j=1}^{m}\alpha_jq_j$, $\alpha_j>0$, by the Gauss-Bonnet formula,
\begin{equation*}
\frac 1{2\pi}\int_{M}K_gdg = \mathcal{X}(M, \A ).
\end{equation*}
Here $\mathcal{X}(M, \A)$ is the Euler characteristic of $(M,\A)$ defined by
$$\mathcal{X}(M, \A)= \mathcal{X}(M) + |\A|,$$
where $\mathcal{X}(M)=2-2g_M$ is the topological Euler characteristic of $M$ itself, $g_M$ is the genus of $M$ and $|\A|= \Sigma_{j=1}^{m}\alpha_j$ is the degree of $\A$. Then we deduce that
\begin{equation*}
\int_{M}2e^{2u_n}-e^{u_n}|\psi_n|^2dg=\int_{M}2e^{2(u_n+\phi)}-e^{u_n+\phi}|e^{\frac{\phi}2}\psi_n|^2dg_0
=4\pi(1-g_M)+2\pi\Sigma_{j=1}^{m}\alpha_j.
\end{equation*}Since the possible values of $\lim_{n\rightarrow \infty} \int_{M}2e^{2(u_n+\phi)}-e^{u_n+\phi}|e^{\frac{\phi}2}\psi_n|^2dg_0$ are $$4\pi k_0+\Sigma_{j=1}^{m}4\pi(1+\alpha_j)k_j$$ for some nonnegative integers $k_0$ and $k_j, j=1,...,m$. Therefore we have the following:

\begin{thm}\label{global-blowup}
Let $(M,\A,g)$ be a surface with divisor $\A=\Sigma_{j=1}^{m}\alpha_j q_j, \alpha_j>0$.  Then
\begin{itemize}
\item[(i)] if $4\pi(1-g_M)+2\pi\Sigma_{j=1}^{m}\alpha_j = 4\pi$, then the blow-up set $\Sigma_1$ contains at most one point. In particular, $\Sigma_1$ contains at most one point if $g_M=0$ and $\A=0$.
\item[(ii)] if $4\pi(1-g_M)+2\pi\Sigma_{j=1}^{m}\alpha_j < 4\pi$, then the blow-up set $\Sigma_1=\emptyset$.
\end{itemize}
\end{thm}

\begin{rem}
Our method can also be applied to deal with a sequence of solutions  $(u_n,\Psi _n)$ of the following local super-Liouville type equations
with two coefficient functions
\begin{equation} \label{Eq-Sn2}
\left\{
\begin{array}{rcl}
-\Delta u_n(x) &=& 2V_n^2(x)|x|^{2\alpha_n}e^{2u_n(x)}-W_n(x)|x|^{\alpha_n}e^{u_n(x)}|\Psi_n|^2  \quad\\
\slashiii{D}\Psi_n &=& -W_n(x)|x|^{\alpha_n}e^{u_n(x)}\Psi_n \end{array} \text {in } B_{r},
\right.
\end{equation}
and satisfying the energy condition
\begin{equation}\label{Eq-Cn2}
\int_{B_r}|x|^{2\alpha_n}e^{2u_n}dx<C,\text{ and }\int_{B_r}\left| \Psi _n\right|
^4dx<C.
\end{equation}
for some constant $C>0$, where
\begin{itemize}
\item[i)]  $\alpha_n > -1$ and $ \alpha_n \rightarrow \alpha > -1$,
\item[ii)] $V_n, W_n \in C^{0}(\overline{B_r}), 0< a\leq V_n(x), W_n(x) \leq b < + \infty$,
$||\nabla V_n||_{L^{\infty}(\overline{B_r})}+||\nabla W_n||_{L^{\infty}(\overline{B_r})} \leq C $.
\end{itemize}
By slightly modifying the proofs of some analytical properties in Section \ref{local}, Section \ref{pohozaev}, Section \ref{bubble} as well as  Theorem  \ref{mainthm}, Theorem \ref{engy-indt}, Theorem \ref{mainthm1}, Theorem \ref{BV}, Theorem \ref{thmsin},
the corresponding blow-up results hold (see more details in Section \ref{two}).
For similar results for Liouville type equations with singular data and with $-1< \alpha <0$, we refer to  \cite{BaMo}.
\end{rem}

\section*{\bf{Acknowledgements}} The research leading to these results has received funding from the
European Research Council under the European Union's Seventh
Framework Programme (FP7/2007-2013) / ERC grant agreement
No. 267087. The second author is also supported partially by NSFC of China (No. 11271253). The third author was supported in part by National Science Foundation of China (No. 11601325).

\
\

\section{Invariance of the global system and special solutions}\label{conf}

In this section, we start with the  invariance of the global super-Liouville equations under conformal diffeomorphisms that preserve the conical points. Then, we shall provide two special solutions.

\begin{prop}\label{prop-1}
The functional $E(u,\psi )$ is  invariant under conformal diffeomorphisms $\varphi :M\rightarrow M$ preserving the divisor, that is,  $\varphi^*\A=\A$ and $\varphi^*(ds^2)=\lambda ^2 ds^2$,
where $\lambda>0 $ is the conformal factor of the conformal map $\varphi,$. Set
\begin{equation*}\ba{rcl}
\ds\vs \widetilde{u} &=&\ds u\circ \varphi -\ln \lambda  \\
\ds \widetilde{\psi } &=&\ds \lambda ^{-\frac 12}\psi \circ \varphi \ea
\end{equation*}
Then $E(u,\psi )=E(\widetilde{u},\widetilde{\psi })$.
In particular, if $(u,\psi)$ is a solution of
(\ref{eq-2}), so is $(\tilde u,\tilde \psi)$.
\end{prop}

The proof of proposition \ref{prop-1} is the same as that of the case of $\A=0$ considered in \cite{JWZ}.

\

As we will see later (Section 6), however, the local super-Liouville type system (\ref{Eq-LL}) we shall study is not conformally invariant near the conical singularity.
During the blow-up process, after suitable rescaling and translation in the domain, we can obtain bubbling solutions of (\ref{Eq-LL}) and (\ref{Co-LL}) on $\R^2$
with $\alpha\geq 0$ and  $V\equiv const $, at which point we can apply the above invariance of the global system
and the singularity removability results in Section 4 and Section 5 to conclude that these bubbling solutions can be conformally extended to $S^2$.

\

Now we present some examples of solutions of the super-Liouville equations (\ref{eq-2}).
Let  $(M,ds^2)$ be the mathematical version of  an American football, i.e., $M$ is a sphere with two antipodal singularities of equal angle. From \cite{T2},
$(M,ds^2)$  is conformally equivalent to $\C\cup \infty$ with constant curvature $K=1$ and conical
singularities at $z=0$ and $z=\infty$ with the same angle $\alpha$, and with the conformal metric
$$\frac{(2+2\alpha)^2|z|^{2\alpha}dz^2}{(1+|z|^{2+2\alpha})^2}$$
for $\alpha$ being not an integer. Therefore, if we define a conformal map $\varphi:(M,ds^2)\rightarrow  \C\cup \infty$ such that
$$(\varphi^{-1})^*(ds^2)=\frac{(2+2\alpha)^2|z|^{2\alpha}dz^2}{(1+|z|^{2+2\alpha})^2},$$ then $u=\frac 12
\log \frac 12+\frac 12\log \det \left| d\varphi \right|$ are  solutions of
\begin{equation*}
-\Delta u+1-2e^{2u}=0 \qquad \text{ on }  M\backslash\{\varphi^{-1}(0),\varphi^{-1}(\infty)\}.
\end{equation*}
In particular, this  yields solutions of the form $(u,0)$ of (\ref{eq-2}).

\

There is another example of a solution of (\ref{eq-2}). Let us recall
that a {\it Killing spinor} is a spinor $\psi$  satisfying
\[
\nabla _X\psi = \lambda X\cdot \psi , \quad \hbox{
for any vector field } X
\]
for some constant $\lambda$. On the standard sphere, there are Killing spinors
with the Killing constant $\lambda= \frac 12$, see for instance \cite{baum}.
 Such a
Killing spinor
is an eigenspinor, i.e.
\[
\slashiii{D}\psi =- \psi,
\]
with constant $|\psi|^2$. Choosing a Killing spinor $\psi$ with
$|\psi|^2=1$, $(0,\psi)$ is a solution
of (\ref{eq-2}). If we let $\pi$ be the stereographic projection from $\S^2\backslash\{north pole\}$
to the Euclidean plane $\R^2$ such that the metric of $\R^2$ is  \[\frac 4{(|1+|x|^2)^2}|dx|^2,\]
then any Killing spinor has the form
\[\frac{v+x\cdot v}{\sqrt{1+|x|^2}},\]
up to a translation or a dilation (see \cite{baum}). We put $\widetilde \psi=\frac{v+x\cdot v}{\sqrt{1+|x|^2}}$.
Then $(0,\psi)=(0,(\log\det|d\varphi|)^{-\frac 12}\widetilde\psi\circ\varphi)$ is a solution of (\ref{eq-2}).

\
\

\section{The local super-Liouville system}\label{local}

In this section, we shall first derive the local version of the super-Liouville equations. Then we shall analyze the regularity of solutions under the small energy condition. Consequently, we can prove Theorem \ref{mainthm}.

It is well known that (see e.g. \cite{T1}), in a small neighborhood $U(p)$ of a given point $p\in M$, we can define an isothermal coordimate system $x=(x_1,x_2)$ centered at $p$, such that $p$ corresponds to $x=0$ and
$ds^2=e^{2\phi}|x|^{2\alpha}(dx_1^2+dx_2^2)$ in $B_{2r}(0)=\{(x_1^2+x_2^2)< 2r\}$, where $\phi$ is smooth away from $p$ and
continuous at $p$. We can choose such a neighborhood small enough so that  if $p$ is a conical singular point of $ds^2$,
then $U(p)\cap {\A}=\{p\}$ and $\alpha>0 $, while, if $p$ is a smooth point of $ds^2$, then  $U(p)\cap {\A}=\emptyset$
and $\alpha=0 $. Consequently, with respect to the isothermal coordinates, $(u,\psi)$ satisfies
\begin{equation}
\left\{
\begin{array}{rcl}
-\Delta u(x) &=&  e^{2\phi(x)}|x|^{2\alpha}(2e^{2u(x)}-e^{u(x)}|\psi|^2(x) -K_g)\quad\\
\slashiii{D}(e^{\frac {\phi(x)}2}|x|^{\frac {\alpha}2}\psi) &=& -e^{\phi(x)}|x|^{\alpha}e^{u(x)}(e^{\frac {\phi(x)}2}
|x|^{\frac {\alpha}2}\psi)
\end{array} \text {in } B_{r}(0).
\right.   \label{eq-5}
\end{equation}
Here $\Delta=\partial^2_{x_1x_1}+\partial^2_{x_2x_2}$ is the usual Laplacian. The Dirac operator
$\slashiii{D} $ is the usual one, which can be seen as the (doubled) Cauchy-Riemann operator. That is, let
$e_1=\frac{\partial}{\partial x_1}$ and
$e_2=\frac{\partial}{\partial x_2}$ be the standard orthonormal
frame on $\R^2$. A spinor field is simply a map $\Psi:\R^2\to
\Delta_2=\C^2$, and $e_1$, $e_2$ acting on spinor fields can
be identified with multiplication with matrices
\[e_1=\left(\begin{matrix}0& 1\\ -1&0 \end{matrix}\right),
\quad e_2=\left(\begin{matrix}0& i\\ i&0 \end{matrix}\right).\] If
$\Psi:=\left(\begin{matrix} \ds f
\\ \ds g\end{matrix}\right)
:\R^2\to \C^2$ is a spinor field, then the Dirac operator is
\[\slashiii{D}\Psi=\ds \left(\begin{matrix}0& 1\\ -1&0
\end{matrix}\right) \left(\begin{matrix} \ds \frac{\partial
f}{\partial x_1}\\ \ds \frac{\partial g}{\partial x_1}
\end{matrix}\right)+
\left(\begin{matrix}0& i\\ i&0 \end{matrix}\right)
\left(\begin{matrix} \ds \frac{\partial f}{\partial x_2} \\
\ds\frac{\partial g}{\partial x_2}
\end{matrix}\right)=
2\left(\begin{matrix} \ds \frac{\partial g}{\partial \bar z}
\\ -\ds\frac{\partial f}{\partial z}\end{matrix}\right),\]
where
\[\frac{\partial}{\partial z}=\frac 12 \left(\frac{\partial }{\partial x_1}
- i\frac{\partial }{\partial x_2}\right), \quad
\frac{\partial}{\partial \bar z}=\frac 12 \left(\frac{\partial
}{\partial x_1} + i\frac{\partial }{\partial x_2}\right).\]
For more details on Dirac operator and spin geometry,  we refer to \cite{LM}.

We note that the last term in the first equation of \eqref{eq-5} is $e^{2\phi}|x|^{2\alpha}K_g$, which satisfies
$$-\Delta\phi= e^{2\phi}|x|^{2\alpha}K_g. $$
Since $\phi $ is continuous,  elliptic regularity implies that $\phi\in W^{2,p}_{loc}$ for all $p<+\infty$ if $\alpha\geq 0$ and
 if the curvature $K_g$ of $M$ is regular enough. Therefore, by Sobolev embedding, $\phi\in C^{1,\delta}$ if $\alpha\geq 0$.
If we denote $V(x)=e^{\phi}$ and $W(x)=e^{2\phi}|x|^{2\alpha}K_g$, then  $0< a\leq V(x)\leq b$ and $W(x)$ is in $L^p(B_r(0))$ for
all $p>1$ if the curvature $K_g$ of $M$ is regular enough.

\
\
Therefore,  the equations (\ref{eq-5}) can be rewritten as:
\begin{equation*}
\left\{
\begin{array}{rcl}
-\Delta u(x) &=& 2V^2(x)|x|^{2\alpha}e^{2u(x)}-V(x)|x|^{\alpha}e^{u(x)}|\Psi|^2-W(x)  \quad\\
\slashiii{D}\Psi &=& -V(x)|x|^{\alpha}e^{u(x)}\Psi
\end{array} \text {in } B_{r}(0).
\right.
\end{equation*}
Here $\alpha \geq 0$, $V(x)$ and $W(x)$ satisfy the following conditions:
\begin{itemize}
 \item[i)]  $0< a\leq V(x)\leq b$;
\item[ii)] $W(x)  \in L^p(B_r(0))$, for all $p>1$.
\end{itemize}

Furthermore, let $w(x)$ satisfy
\begin{equation*}
\left\{
\begin{array}{rcll}
-\Delta w(x) &=& -W(x)  &\quad \text {in } B_{r}(0),\\
w(x) &=& 0   &\quad \text{ on } \partial B_r(0).
\end{array}
\right.
\end{equation*}
It is easy to see that $w(x)$ is $C^{1,\beta}$ in $B_{r}(0)$ for some $0<\beta<1$. Setting $v(x)=u(x)-w(x)$, then $(v,\Psi)$ satisfies

\begin{equation*}
\left\{
\begin{array}{rcl}
-\Delta v(x) &=& 2V^2(x)e^{2w(x)}|x|^{2\alpha}e^{2v(x)}-V(x)e^{w(x)}|x|^{\alpha}e^{v(x)}|\Psi|^2  \quad\\
\slashiii{D}\Psi &=& -V(x)e^{w(x)}|x|^{\alpha}e^{v(x)}\Psi
\end{array} \text {in } B_{r}(0).
\right.
\end{equation*}

\noindent Now we come to the local version of the  singular super-Liouville-type equations
\begin{equation}\label{Eq-L}
\left\{
\begin{array}{rcl}
-\Delta u(x) &=& 2V^2(x)|x|^{2\alpha}e^{2u(x)}-V(x)|x|^{\alpha}e^{u(x)}|\Psi|^2  \quad\\
\slashiii{D}\Psi &=& -V(x)|x|^{\alpha}e^{u(x)}\Psi
\end{array} \text { in } B_{r}(0),
\right.
\end{equation}
Here $V(x)$ is a $C^{1,\beta}$ function and satisfies $0< a\leq V(x)\leq b$. We also assume that $(u,\Psi)$ satisfy the energy condition:
\begin{equation}\label{Co-L}
\int_{B_r(0)}|x|^{2\alpha}e^{2u}+|\Psi|^4dx<+\infty.
\end{equation}

\
\

Next we consider the regularity of solutions under the energy condition. We put $B_r:=B_r(0)$.

First, we define weak solutions of \eqref{Eq-L} and \eqref{Co-L}. We say that $(u, \Psi)$ is a weak solution of  \eqref{Eq-L} and \eqref{Co-L}, if $u\in W^{1,2}(B_r)$
and $\Psi \in W^{1,\frac 43}(\Gamma (\Sigma B_r))$ satisfy
\begin{eqnarray*}
\int_{B_r} \nabla u\nabla\phi dx &=& \int_{B_r} (2V^2(x)|x|^{2\alpha}e^{2u}-V(x)|x|^{\alpha}e^u |\Psi|^2)\phi dx, \\
\int_{B_r} \langle \Psi,\slashiii{D} \xi \rangle dx &=&-\int_{B_r} V(x)|x|^{\alpha}e^u \langle \Psi,\xi \rangle dx,
\end{eqnarray*}
for any $\phi\in C^\infty_0(B_r)$ and any spinor $\xi \in C^\infty \cap  W_0^{1,\frac 43} (\Gamma (\Sigma B_r))$.
A weak solution is a classical solution by the following:

\begin{prop}\label{prop-a} Let $(u,\Psi)$ be a weak solution  of \eqref{Eq-L} and \eqref{Co-L}.
Then $(u,\Psi)\in C^2(B_r)\times C^2(\Gamma(\Sigma B_r))$.
\end{prop}

Note that when $\alpha=0$ this proposition is proved in \cite{JWZ} (see Proposition 4.1). When $\alpha>0$,
it is clear that we can no longer  use the inequality $2\int {u^+}<\int {e^{2u}}<\infty$  to get  the $L^1$ integral of $u^+$.
So, we need a trick, which was introduced in \cite{bt}, to prove this proposition.

\
\

\noindent{\bf Proof of Proposition \ref{prop-a}}:  By the standard elliptic method,
to prove this propositon, it is  sufficient to show that $u^+\in L^\infty(B_{\frac r4}),\quad |\Psi| \in L^{\infty}(B_{\frac r4}).$

In fact, for the regularity of $u$, let us set
\[
f_1=2V^2(x)|x|^{2\alpha}e^{2u(x)}-V(x)|x|^{\alpha}e^{u(x)}|\Psi|^2.
\]
Then we have
\[
-\Delta u=f_1.
\]
We consider the following Dirichlet problem
\begin{equation}
\left\{
\begin{array}{rcl}
-\Delta u_1 &=& f_1,\qquad \text{in  }B_r \\
 u_1&=& 0,
\qquad \text{ on  }\partial B_r.
\end{array}
\right.  \label{57}
\end{equation}
It is clear that $f_1\in L^1(B_r).$  In view of  Theorem 1 in \cite{BM} we have
\begin{equation}
e^{k\left| u_1\right| }\in L^1(B_r)  \label{66}
\end{equation}
for some $k>1$ and in particular $ u_1\in {L^p(B_r)}$ for
some $p> 1.$

Let $u_2=u-u_1$ so that $\Delta u_2=0$ on $B_r.$ The mean value
theorem for harmonic functions implies that
\[
\left\| u_2^{+}\right\| _{L^\infty (B_{\frac r2})}\leq C\left\|
u_2^{+}\right\| _{L^1(B_r)}.
\]
On the other hand, it is clear that for some $t>0$,
$$
\int_{B_r(0)}\frac{1}{|x|^{2t\alpha}}dx\leq C.
$$
Hence we can choose $s=\frac t{t+1}\in (0,1)$ when $\alpha>0$ and $s=1$ when $\alpha=0$ such that
$$
2s\int_{B_r}u^+dx\leq \int_{B_r}e^{2su}dx\leq (\int_{B_r}|x|^{2\alpha}e^{2u}dx)^s
(\int_{B_r}|x|^{-2t\alpha}dx)^{1-s}<\infty.
$$
Then by using $u_2^{+}\leq u^{+}+\left| u_1\right|$ we obtain that
$u_2^{+}\in{L^1(B_r)}$ and consequently

\begin{equation}
\left\| u_2^{+}\right\| _{L^\infty (B_{\frac r2})}<\infty.
\label{55}
\end{equation}
Next we rewrite $f_1$ as
\[
f_1=2V^2(x)|x|^{2\alpha}e^{2u_2(x)}e^{2u_1(x)}-V(x)|x|^{\alpha}e^{u_2(x)}e^{u_1(x)}|\Psi|^2
\]
From (\ref{66}) and (\ref{55}) we have $f_1\in L^{1+\varepsilon
}(B_{\frac r2})$ for some $\varepsilon
>0.$ Hence the standard elliptic estimates imply that

\[
\left\| u^{+}\right\| _{L^\infty (B_{\frac r4})}\leq C\left\|
u^{+}\right\| _{L^1(B_r)}+C\left\| f_1\right\| _{L^{1+\varepsilon
}(B_{\frac r2})}<\infty .
\]

Since $u^+ \in L^{\infty}(B_{\frac r4})$, then the right hand of
equation $\slashiii{D} \Psi =-V(x)|x|^{\alpha}e^u\Psi$ is in $L^4(\Gamma (\Sigma
B_{\frac r4}))$. Hence $\Psi \in C^0(\Gamma (\Sigma B_{\frac
r4}))$ and especially $|\Psi| \in L^{\infty}(B_{\frac r4})$.
\qed

\
\

Next we  discuss the blow-up behavior of a sequence of solutions $(u_n,\Psi_n)$ satisfying (\ref{Eq-Sn}) and (\ref{Eq-Cn}). First, we study the small energy regularity, i.e. when the energy $\int_{B_r}|x|^{2\alpha_n}e^{2u_n}dx$ is small enough,
$u_n$ will be uniformly bounded from above. Our Lemma is:
\begin{lm}\label{lmuni}
Let $0<\varepsilon _0<\pi$ be a constant.  For any
 sequence of solutions $(u_n,\Psi_n)$ to (\ref{Eq-Sn}) with
\[
\int_{B_r}|x|^{2\alpha_n}e^{2u_n}dx<\varepsilon _0,\qquad \int_{B_r}\left| \Psi
_n\right| ^4dx<C
\]
for some fixed constant $C>0$, we have that $\left\| u_n^{+}\right\|_{L^{\infty}
(B_{\frac r4})}$ is uniformly bounded.
\end{lm}
\begin{proof} We are in the same situation as in Proposition \ref{prop-a}. When $\alpha_n>0$, we can no longer
use the inequality $2\int {u^+_n}<\int {e^{2u_n}}$  to get  the uniform bound of the $L^1$-integral of $u^+_n$. But notice that there exists a uniform constant $t>0$ such that for all $n$
$$
\int_{B_r}\frac{1}{|x|^{2t\alpha_n}}dx\leq C,
$$  since $\alpha_n\rightarrow \alpha$ and $\alpha\geq 0$. Consequently we obtain $s=\frac t{t+1}\in (0,1)$
$$
2s\int_{B_r}u^+_ndx\leq \int_{B_r}e^{2su_n}dx\leq (\int_{B_r}|x|^{2\alpha_n}e^{2u_n}dx)^s
(\int_{B_r}|x|^{-2t\alpha_n}dx)^{1-s}<C.
$$
Then by a similar argument as in the proof of Lemma 4.4 in \cite{JWZ} we can prove this Lemma.
\end{proof}

\

When the energy $\int_{B_r}|x|^{2\alpha_n}e^{2u_n}dx$ is large, the blow-up phenomenon may occur as in the case of a smooth domain.

\

\noindent{\bf Proof of Theorem \ref{mainthm}}: By using Lemma \ref{lmuni} and applying a similar argument as in the proof of Theorem 5.1 in \cite{JWZ}, we can easily prove this theorem.\qed

\

\begin{rem}\label{rem3.4}Let $v_n=u_n+\alpha_n\log|x|$, then $(v_n, \Psi_n)$ satisfies
 \begin{equation*}
\left\{
\begin{array}{rcl}
-\Delta v_n(x) &=& 2V^2(x)e^{2v_n(x)}-V(x)e^{v_n(x)}|\Psi_n|^2    - 2 \pi \alpha_n \delta_{p=0} \quad  \\
\slashiii{D}\Psi_n &=& -V(x)e^{v_n(x)}\Psi_n \end{array}     \   \text {in } B_{r},
\right.
\end{equation*}
with the energy condition
\begin{equation*}
\int_{B_r}e^{2v_n}dx<C,\text{ and }\int_{B_r}\left| \Psi _n\right|
^4dx<C.
\end{equation*}
Then the two blow-up sets of $u_n$ and $v_n$ are the same, by using similar arguments as in \cite{bt}.
\end{rem}

\
\

\section{The Pohozaev identity and removability of local singularities}\label{pohozaev}

This section is the heart of our paper. We shall show that a local singularity is removable if and only if the Pohozaev identity is satisfied. To express this result in compact form, we start by defining  a constant that is associated to the equations \eqref{Eq-L} with the constraint \eqref{Co-L}.

\begin{df} \label{poho-const}Let $(u,\Psi)\in C^2(B_r\backslash\{0\})\times C^2(\Gamma(\Sigma (B_r\backslash\{0\})))$ be a solution of \eqref{Eq-L} and \eqref{Co-L}. For $0<R<r$, we define the {\em Pohozaev constant} with respect to the equations \eqref{Eq-L} with the constraint \eqref{Co-L}
 \begin{eqnarray*}
C(u,\Psi)&:=& R\int_{\partial B_R(0)} |\frac {\partial u}{\partial \nu}|^2-\frac 12|\nabla u|^2d\sigma \\
&-&(1+\alpha)\int_{B_R(0)}(2V^2(x)|x|^{2\alpha}e^{2u}-V(x)|x|^{\alpha}e^u|\Psi|^2)dx \\
& & +R\int_{\partial B_R(0)}V^2(x)|x|^{2\alpha}e^{2u}d\sigma-\frac 12\int_{\partial
B_R(0)}\la\frac {\partial \Psi}{\partial \nu}, x\cdot\Psi\ra +\la
x\cdot\Psi, \frac {\partial \Psi}{\partial \nu}\ra d\sigma\\
& & -\int_{B_R(0)}(|x|^{2\alpha}e^{2u}x\cdot\nabla (V^2(x))-|x|^{\alpha}e^u|\Psi|^2x\cdot \nabla
V(x))dx
\end{eqnarray*}
where $\nu$ is the outward normal vector of $\partial B_R(0)$
\end{df}

It is clear that $C(u,\Psi)$ is independent of $R$ for $0<R<r$.

Thus, the vanishing of the Pohozaev constant $C(u,\Psi)$ is equivalent to the {\em Pohozaev identity}
\begin{eqnarray}\label{poho}
&&R\int_{\partial B_R(0)} |\frac {\partial u}{\partial \nu}|^2-\frac 12|\nabla u|^2d\sigma \nn\\
&=&(1+\alpha)\int_{B_R(0)}(2V^2(x)|x|^{2\alpha}e^{2u}-V(x)|x|^{\alpha}e^u|\Psi|^2)dx \nn\\
& & -R\int_{\partial B_R(0)}V^2(x)|x|^{2\alpha}e^{2u}d\sigma+\frac 12\int_{\partial
B_R(0)}(\la\frac {\partial \Psi}{\partial \nu}, x\cdot\Psi\ra +\la
x\cdot\Psi, \frac {\partial \Psi}{\partial \nu}\ra ) d\sigma   \nn\\
& & +\int_{B_R(0)}(|x|^{2\alpha}e^{2u}x\cdot\nabla (V^2(x))-|x|^{\alpha}e^u|\Psi|^2x\cdot \nabla
V(x))dx
\end{eqnarray}
for  a solution $(u,\Psi)\in C^2(B_r)\times C^2(\Gamma(\Sigma B_r))$  of (\ref{Eq-L}) and (\ref{Co-L}).

We can now formulate the main result of this section. This result says that a local singularity is removable iff the Pohozaev identity \eqref{poho} holds, that is, iff the Pohozaev constant vanishes.
\begin{thm}(Removability of a local singularity)\label{thm-sigu-move1}
Let $(u,\Psi)\in C^2(B_r\setminus\{0\})\times C^2(\Gamma(\Sigma (B_r\setminus \{0\})))$  be  a solution of (\ref{Eq-L}) and (\ref{Co-L}).
Then there is a constant $\gamma < 2\pi (1+\a)$ such that
\begin{eqnarray*}
u(x)=- \frac{\gamma}{2\pi}{\rm log} |x| + h, \quad {\rm near }\ 0,
\end{eqnarray*}
where $h$ is bounded near $0$. The Pohozaev constant  $C(u, \Psi) $ and $\gamma$ satisfy:
\begin{eqnarray*}
 C(u, \Psi)=\frac{\gamma^2}{4\pi}.
\end{eqnarray*}
In particular, $(u,\Psi)\in C^2(B_r)\times C^2(\Gamma(\Sigma B_r))$, i.e. the local singularity of $(u,\Psi)$ is removable,   iff $C(u,\Psi)=0$.
\end{thm}

In the remainder of this section, we shall prove the two directions of Theorem \ref{thm-sigu-move1}.
We shall first show that for smooth solutions, the Pohozaev identity \eqref{poho} holds.


\begin{prop}\label{prop-poho}
Let $(u,\Psi)\in C^2(B_r)\times C^2(\Gamma(\Sigma B_r))$ be  a solution of (\ref{Eq-L}) and (\ref{Co-L}). Then,
 for any $0<R<r$,  the Pohozaev type identity \eqref{poho} holds.
\end{prop}

The  case where  $\alpha=0$ and $V\equiv1$ has already been treated in \cite{JWZZ1}.
\
\

\begin{proof}
For $x\in \R^2$, we put  $x=x_1e_1+x_2e_2$. We
multiply all terms in  (\ref{Eq-L}) by $x\cdot \nabla u$ and
integrate over $B_R(0)$. We  obtain
\[\int_{B_R(0)}\Delta u x\cdot \nabla udx = R\int_{\partial B_R(0)}
|\frac {\partial u}{\partial \nu}|^2-\frac 12|\nabla u|^2d\sigma,\]
and
\begin{eqnarray*}
\int_{B_R(0)}2V^2(x)|x|^{2\alpha}e^{2u}x\cdot \nabla udx
&=& R\int_{\partial B_R(0)}V^2(x)|x|^{2\alpha}e^{2u}d\sigma -(2+2\alpha)\int_{B_R(0)}V^2(x)|x|^{2\alpha}e^{2u}dx\\
& &-\int_{B_R(0)}x\cdot \nabla (V^2(x))|x|^{2\alpha}e^{2u}dx,
\end{eqnarray*}
 and
\begin{eqnarray*}
 \int_{B_R(0)}V(x)|x|^{\alpha}e^u|\Psi|^2 x\cdot \nabla udx
&=& R\int_{\partial B_R(0)}V(x)|x|^\alpha e^u|\Psi|^2d\sigma-\int_{B_R(0)}|x|^{\alpha}e^ux\cdot \nabla
(V(x)|\Psi|^2)dx\\
& &-(2+\alpha)\int_{B_R(0)}V(x)|x|^\alpha e^u|\Psi|^2dx.
\end{eqnarray*}
Therefore we get
\begin{eqnarray}\label{4.2}
&&R\int_{\partial B_R(0)} |\frac {\partial u}{\partial \nu}|^2-\frac
12|\nabla u|^2d\sigma \nonumber\\
&=& (2+2\alpha)\int_{B_R(0)}V^2(x)|x|^{2\alpha}e^{2u}dx-(2+\alpha)\int_{B_R(0)}V(x)|x|^\alpha e^u|\Psi|^2dx \nonumber\\
& & -R\int_{\partial B_R(0)}V^2(x)|x|^{2\alpha}e^{2u}d\sigma+R\int_{\partial B_R(0)}V(x)|x|^\alpha e^u|\Psi|^2d\sigma \nonumber\\
& &+\int_{B_R(0)}|x|^{2\alpha}e^{2u}x\cdot \nabla (V^2(x))-|x|^\alpha e^ux\cdot \nabla (V(x)|\Psi|^2)dx.
\end{eqnarray}

On the other hand, by the Schr\"{o}dinger-Lichnerowicz formula $\slashiii{D}^2=-\Delta $ on $\R^2$, we have
\begin{equation}\label{4.4}
\Delta
\Psi=\sum^2_{\alpha=1}\nabla_{e_\alpha}(V(x)|x|^{\alpha}e^u)e_\alpha\cdot\Psi-V^2(x)|x|^{2\alpha}e^{2u}\Psi.
\end{equation}
Here $\cdot$ is the Clifford multiplication and $\{e_1,e_2\}$ is the local orthonormal basis on $\R^2$. Using the Clifford multiplication relation
\[
e_i\cdot e_j+e_j\cdot e_i=-2\delta _{ij},\text{ for }1\leq i,j\leq 2
\] and
\[
\left\langle \Psi ,\varphi \right\rangle =\left\langle e_i\cdot \Psi
,e_i\cdot \varphi \right\rangle
\]
for any spinors $\Psi ,\varphi \in \Gamma (\Sigma M)$, we  know that
\begin{equation}\label{4.3}
\la\Psi,e_i\cdot\Psi\ra +\la e_i\cdot\Psi,\Psi\ra=0
\end{equation}
for any $i=1,2$.
Then we multiply (\ref{4.4}) by $x\cdot\Psi$ and integrate over $B_R(0)$ to obtain
\begin{eqnarray*}&& \int_{B_R(0)}\la\Delta\Psi,x\cdot\Psi\ra dx\\
&=& \int_{B_R(0)}
\sum^2_{\alpha,\beta=1}\la\nabla_{e_\alpha}(V(x)|x|^{\alpha}e^u)e_\alpha\cdot\Psi,e_\beta\cdot\Psi\ra
 x_\beta -V^2(x)|x|^{2\alpha}e^{2u}\la\Psi,x\cdot \Psi\ra dx,
\end{eqnarray*}
and
\begin{eqnarray*}&&\int_{B_R(0)}\la x\cdot \Psi,\Delta\Psi\ra dx\\
&=& \int_{B_R(0)}\sum^2_{\alpha,\beta=1}\la e_{\beta}\cdot
\Psi,\nabla_{e_\alpha}(V(x)|x|^{\alpha}e^u)e_\alpha\cdot\Psi\ra
 x_\beta-V^2(x)|x|^{2\alpha}e^{2u}\la x\cdot \Psi, \Psi\ra dx.
\end{eqnarray*}

\noindent By integration by parts, we get
\begin{eqnarray*}
& & \int_{B_R(0)}\la\Delta\Psi,x\cdot\Psi\ra dx = \int_{\partial B_R(0)}\la\frac {\partial \Psi}{\partial
\nu},x\cdot\Psi\ra d\sigma-\int_{B_R(0)} V(x)|x|^\alpha e^u|\Psi|^2
dx-\int_{B_R(0)}\la\nabla \Psi,x\cdot\nabla \Psi\ra dx,
\end{eqnarray*}
and similarly we have
\begin{eqnarray*}
& & \int_{B_R(0)}\la x\cdot\Psi,\Delta \Psi\ra dx  = \int_{\partial B_R(0)}\la
x\cdot \Psi, \frac {\partial \Psi}{\partial \nu}\ra
d\sigma-\int_{B_R(0)}V(x)|x|^\alpha e^u|\Psi|^2 dx-\int_{B_R(0)}\la x\cdot\nabla\Psi,\nabla \Psi\ra dx.
\end{eqnarray*}

\noindent Furthermore we also have
\begin{eqnarray*}
&&\int_{B_R(0)}\sum^2_{\alpha,\beta=1}\la\nabla_{e_\alpha}(V(x)|x|^\alpha e^u)e_\alpha\cdot\Psi,e_\beta\cdot\Psi\ra
 x_\beta dx\\
&& +\int_{B_R(0)}\sum^2_{\alpha,\beta=1}\la e_{\beta}\cdot
\Psi,\nabla_{e_\alpha}(V(x)|x|^\alpha e^u)e_\alpha\cdot\Psi\ra
 x_\beta dx\\
 &=&2\int_{B_R(0)}\sum^2_{\alpha=1}\la\nabla_{e_\alpha}(V(x)|x|^\alpha e^u)e_\alpha\cdot\Psi,e_\alpha\cdot\Psi\ra
 x_\alpha dx\\
 &=&2\int_{B_R(0)}x\cdot\nabla(V(x)|x|^\alpha e^u)|\Psi|^2dx\\
 &=&-2\int_{B_R(0)}V(x)|x|^\alpha e^u x\cdot\nabla(|\Psi|^2)dx-4\int_{B_R(0)}V(x)|x|^\alpha e^u|\Psi|^2dx
\\& & +2R\int_{\partial
 B_R(0)}V(x)|x|^\alpha e^u|\Psi|^2dx.
\end{eqnarray*}

Therefore we obtain
\begin{eqnarray}\label{4.5}
& & R\int_{\partial
B_R(0)}V(x)|x|^\alpha e^u|\Psi|^2d\sigma-\int_{B_R(0)}V(x)|x|^\alpha e^u x\cdot\nabla(|\Psi|^2)dx \\
&= &\frac 12\int_{\partial B_R(0)}\la\frac {\partial \Psi}{\partial
\nu},x\cdot\Psi\ra d\sigma+\frac 12\int_{\partial B_R(0)}\la
x\cdot\Psi, \frac {\partial \Psi}{\partial \nu}\ra
d\sigma+\int_{B_R(0)} V(x)|x|^\alpha e^u|\Psi|^2 dx. \nonumber
\end{eqnarray}
Combining  (\ref{4.2}) and (\ref{4.5}), we obtain our Pohozaev identity \eqref{poho}.
\end{proof}

Proposition \ref{prop-poho} also shows that $C(u,\Psi)=0$ if $(u,\psi)$ is classical solution of \eqref{Eq-L} with the condition \eqref{Co-L} in $B_r$. For the converse, let us start with  a lemma.

\begin{lm}\label{asy-phi} There exists $0<\varepsilon_0 < \pi$ such that if $(v,\phi)$ is a solution of
\begin{equation*}
\left\{
\begin{array}{rcl}
-\Delta v &=& 2h^2(x)|x|^{2\alpha}e^{2v}-h(x)|x|^{\alpha}e^v\left\langle \phi ,\phi
\right\rangle,\\
\slashiii{D}\phi &=&\ds  -h(x)|x|^{\alpha}e^v\phi,
\end{array}\qquad x\in B_{r_0}\backslash\{0\},
\right.
\end{equation*}
where $h(x)$ is a $C^{1,\beta}$ function satisfying $0< a\leq h(x)\leq b$ in $B_{r_0}$ and it satisfies
$$\int_{B_{r_0}}|x|^{2\alpha}e^{2v}dx<\varepsilon_0, \quad \int_{ B_{r_0}}|\phi|^4dx<C,$$
then for any $x\in B_{\frac {r_0}{2}}$ we have
\begin{equation*}
|\phi(x)||x|^{\frac 12}+|\nabla\phi(x)||x|^{\frac 32}\leq
C(\int_{B_{2|x|}}|\phi|^4dx)^{\frac 14}.
\end{equation*}
Furthermore, if we assume that $e^{2v}=O(\frac
{1}{|x|^{2+2\alpha-\varepsilon}})$, then,  for any $x\in B_{\frac {r_0}2}$, we
have
\begin{equation*}
|\phi(x)||x|^{\frac 12}+|\nabla\phi(x)||x|^{\frac 32}\leq
C|x|^{\frac {1}{4C}}(\int_{B_{r_0}}|\phi|^4dx)^{\frac 14},
\end{equation*}
for some positive constant $C$. Here $\varepsilon$ is any
sufficiently small positive number.
\end{lm}

\begin{proof}
Set $ w(x)=v(x)+\alpha \ln|x|$. Then $(w,\phi)$ satisfies
\begin{equation*}
\left\{
\begin{array}{rcl}
-\Delta w &=& 2h^2(x)e^{2w}-h(x)e^w\left\langle \phi ,\phi
\right\rangle,\\
\slashiii{D}\phi &=&\ds  -h(x)e^v\phi,
\end{array}\qquad x\in
B_{r_0}\backslash\{0\},
\right.
\end{equation*}
with the energy conditions
\begin{eqnarray*}
 \int_{B_{r_0}}e^{2w}dx\leq \varepsilon_0, \quad \int_{B_{r_0}}|\phi|^4(x)dx\leq C.
\end{eqnarray*}
Since $h(x)$ is a $C^{1,\beta}$ function satisfying $0< a\leq h(x)\leq b$ in $B_{r_0}$, we can obtain the conclusion of this lemma by applying similar arguments as in the proof of Lemma 6.2 in \cite{JWZ}.
\end{proof}

\

We shall now show the removability of a local singularity when the Pohozaev constant vanishes, thereby completing the proof of Theorem \ref{thm-sigu-move1}.
\begin{prop}(Removability of a local singularity)\label{sigu-move1}
Let $(u,\Psi)\in C^2(B_r\setminus\{0\})\times C^2(\Gamma(\Sigma (B_r\setminus \{0\})))$  be  a solution of (\ref{Eq-L}) and (\ref{Co-L}).
Then there is a constant $\gamma < 2\pi (1+\a)$ such that
\begin{eqnarray*}
u(x)=- \frac{\gamma}{2\pi}{\rm log} |x| + h, \quad {\rm near }\ 0,
\end{eqnarray*}
where $h$ is bounded near $0$. Moreover, the Pohozaev constant  $C(u, \Psi) $ and $\gamma$ are related by
\begin{eqnarray*}
 C(u, \Psi)=\frac{\gamma^2}{4\pi}.
\end{eqnarray*}
In particular, if $C(u,\Psi)=0$, then $(u,\Psi)\in C^2(B_r)\times C^2(\Gamma(\Sigma B_r))$, i.e. the local singularity of $(u,\Psi)$ is removable.
\end{prop}

\begin{proof} Since $\int_{B_1}|x|^{2\alpha}e^{2u}dx=\int_{B_r}|x|^{2\alpha}e^{2\widetilde{u}}dx$  under the following scaling transformation
\begin{eqnarray*}
& & \widetilde{u}(x)=u(rx)-(1+\alpha)\ln r,\\
& & \widetilde{\Psi}(x)=r^{-\frac 12} \Psi (rx),
\end{eqnarray*}
we assume for convenience that $\int_{B_r}|x|^{2\alpha}e^{2u}dx <\varepsilon_0$,
where $\varepsilon_0$ is as in Lemma \ref{asy-phi}.
By standard potential analysis, it follows that there is a
constant $\gamma$ such that
\begin{equation*}
\lim_{|x|\rightarrow 0}\frac {u}{-\log|x|}=\frac{\gamma}{2\pi}.
\end{equation*}
By $\int_{B_r}|x|^{2\alpha}e^{2u}+|\Psi|^4dx<C$ we obtain that $\gamma\leq 2\pi(1+\alpha)$.
Furthermore, by using Lemma \ref{asy-phi} and by a similar argument as in the proof of Proposition 2.6 of
\cite{JWZZ1}, we can improve  this to the strict inequality $\gamma<2\pi(1+\alpha)$.

Define $v(x)$ by
$$
v(x)=-\frac 1{2\pi}\int_{B_r}\log|x-y|(2V^2(y)|y|^{2\alpha}e^{2u}-V(y)|y|^{\alpha}e^u|\Psi|^2)dy
$$
and set $w=u-v$. It is clear that $-\Delta v=2V^2(x)|x|^{2\alpha}e^{2u}-V(x)|x|^{\alpha}e^u|\Psi|^2$
in $B_r$ and $\Delta w=0$ in $B_r\backslash \{0\}$. One can check
that
$$\lim_{|x|\rightarrow 0}\frac {v(x)}{-\log|x|}=0$$ which implies that
$$
\lim_{|x|\rightarrow 0}\frac {w(x)}{-\log|x|}=\lim_{|x|\rightarrow
0}\frac{u-v}{-\log|x|}=\frac{\gamma}{2\pi}.
$$
Since $w$ is harmonic in $B_1\backslash \{0\}$ we have
$$
w=-\frac{\gamma}{2\pi}\log|x|+w_0
$$ with a smooth harmonic function $w_0$ in $B_r$. Therefore we have
$$
u=-\frac{\gamma}{2\pi}\log|x|+v+w_0 \quad \text{ near } 0.
$$
Next we will compute the Pohozaev constant for $(u,\Psi)$. For this purpose, we want to estimate the decay
of $(v,\Psi)$ near the zero. Since
$$
-\Delta v=2V^2(x)|x|^{2\alpha}e^{2u}-V(x)|x|^{\alpha}e^u|\Psi|^2,
$$
and the right hand term $f_1(x):=2V^2(x)|x|^{2\alpha}e^{2u(x)}$ and $ f_2(x):=-V(x)|x|^{\alpha}e^{u(x)}|\Psi|^2(x) $ are $L^1$ integrable, we can obtain  $e^{|v(x)|}\in L^p(B_r)$ for any $p\geq 1$. Since
$$
f_1(x)=|x|^{-\frac {\gamma}{\pi}+2\alpha }(2V^2(x)e^{2w_0(x)+2v(x)})
$$ and
$$
f_2(x)=-|x|^{-\frac {\gamma}{2\pi}+\alpha-1}(V(x)e^{w_0(x)+v(x)}|x||\Psi|^2(x)),
$$
we set $s_1=\frac {\gamma}{\pi}-2\alpha $ and  $s_2=\frac {\gamma}{2\pi}-\alpha+1 $. Then $\max\{s_1, s_2\}<2$.
Since $|\Psi|\leq C|x|^{-\frac 12}$ near $0$ and $w_0(x)$ is smooth in $B_r$, we have by H\"older's inequality that $f_1\in L^{t}(B_r)$ for any $t\in (1,\frac 2{s_1})$ if $s_1>0$, and $f_1\in L^{t}(B_r)$ for any $t>1$ if $s_1\leq 0$. For $f_2$,  we also have $f_2\in L^t(B_r)$ for any $t\in (1,\frac 2{s_2})$ if $s_2>0$, and $f_2\in L^{t}(B_r)$ for any $t>1$ if $s_2\leq 0$. Altogether, there exists some $t>1$ such that $f\in L^t(B_r)$. In turn, we get that $v(x)$ is in $L^{\infty}(B_r)$.
On the other hand, since $v(x)$ is in $L^{\infty}(B_r)$, it follows from Lemma  \ref{asy-phi} that there exists a small $\delta_0>0$ such that
$$
|\Psi|\leq C|x|^{\delta_0-\frac 12}, \quad \text{ near } 0,
$$
and
$$
|\nabla \Psi|\leq C|x|^{\delta_0-\frac 32}, \quad \text{ near } 0.
$$
Next we estimate $\nabla v(x)$. If $s_1<0$ and $s_2<0$, then $v(x)$ is in $C^1(B_r)$.
If $s_1>0$ or $s_2>0$, $\nabla v(x)$ will have a decay when $|x|\rightarrow 0$. Without loss of generality, we assume that $s_1>0$ and $s_1>0$.  Denote
$$
v_1(x)=-\frac 1{2\pi}\int_{B_r}(\log|x-y|)(2V^2(y)|y|^{2\alpha}e^{2u(y)})dy,
$$
and
$$
v_2(x)=\frac 1{2\pi}\int_{B_r}(\log|x-y|)(V(y)|y|^{\alpha}e^{u(y)}|\Psi|^2(y))dy.
$$
Note that
\begin{eqnarray*}
|\nabla v_1(x)| &\leq & \frac{1}{2\pi}\int_{B_r}\frac{1}{|x-y|}|f_1(y)|dy\\
 & = & \frac{1}{2\pi}\int_{\{|x-y|\geq\frac {|x|}2\}\cap {B_r}}\frac{1}{|x-y|}|f_1(y)|dy+ \frac{1}{2\pi}\int_{\{|x-y|\leq\frac {|x|}2\}\cap {B_r}}\frac{1}{|x-y|}|f_1(y)|dy\\
 & = & I_1+I_2.
\end{eqnarray*}

Fix $t\in (1,\frac 2{s_1})$ and choose $0<\tau_1<1$ such that $\frac {\tau_1 t}{t-1}<2$. Hence, we have $0<\tau_1 <2-s_1$. Then by H\"{o}lder's inequality we obtain
\begin{eqnarray*}
I_1 & \leq & (\int_{\{|x-y|\geq\frac {|x|}2\}\cap {B_r}}\frac 1{|x-y|^{\frac {\tau_1t}{t-1}}}dy)^{\frac {t-1}t}(\int_{\{|x-y|\geq\frac {|x|}2\}\cap {B_r}}\frac{1}{|x-y|^{(1-\tau_1)t}}|f_1|^tdy)^{\frac 1t}\\
& \leq & \frac {C}{|x|^{1-\tau_1}}.
\end{eqnarray*}
For $I_2$, since $y\in\{y||x-y|\leq \frac{|x|}2\}$ implies that $|y|\geq \frac{|x|}2$, we can get that
\begin{eqnarray*}
I_2 &\leq & C \int_{\{|x-y|\leq\frac {|x|}2\}\cap {B_r}}\frac{1}{|x-y||y|^{s_1}}dy \nonumber \\
&\leq & C|x|^{1-s_1}
\end{eqnarray*}
Hence we have
$$
|\nabla v_1(x)|\leq C(\frac 1{|x|^{1-\tau_1}}+|x|^{1-s_1})
$$ for suitable $\tau_1\in (0,2-s_1)$.
Similarly, we also can get that
$$
|\nabla v_2(x)|\leq C(\frac 1{|x|^{1-\tau_2}}+|x|^{1-s_2})
$$ for suitable $\tau_2\in (0,2-s_2)$. In conclusion, we have

$$
|\nabla v(x)|\leq C(\frac 1{|x|^{1-\tau}}+|x|^{1-s})
$$ for suitable $\tau=\min\{\tau_1,\tau_2\}$ and $s=\max\{s_1,s_2\}$.

Now, we can compute the Pohozaev constant $C(u, \Psi)$. Since $\nabla u=-\frac{\gamma}{2\pi}\frac x{|x|^2}+\nabla (w_0+v(x))$, we have for any $0<R<r$
\begin{eqnarray*}
&& R\int_{\partial B_R}|\frac {\partial u}{\partial \nu}|^2-\frac 12|\nabla u|^2d\sigma\\
&=& R\int_{\partial B_R}(\frac x{|x|}\cdot\nabla(w_0+v)-\frac{\gamma}{2\pi}\frac 1{|x|})^2-\frac 12(\frac{\gamma^2}{4\pi^2}\frac 1{|x|^2}-2\frac \gamma{2\pi}\frac x{|x|^2}\cdot \nabla (w_0+v)+|\nabla (w_0+v)|^2) d\sigma\\
&=&\frac 1{4\pi}\gamma^2-\frac \gamma{2\pi}R\int_{\partial B_R}\frac x{|x|^2}\cdot \nabla(w_0+v)d\sigma-\int_{\partial B_R}\frac R2 |\nabla (w_0+v)|^2-R(\frac x{|x|}\cdot \nabla (w_0+v))^2d\sigma\\
&=&\frac 1{4\pi}\gamma^2+o_R(1).
\end{eqnarray*} where $o_R(1)\rightarrow 0$ as $R\rightarrow 0$.
We also have
$$
(1+\alpha)\int_{B_R}2V^2(x)|x|^{2\alpha}e^{2u}-V(x)|x|^{\alpha}e^{u}|\Psi|^2 dx =o_R(1),
$$
and
$$
 R\int_{\partial B_R}V^2(x)|x|^{2\alpha}e^{2u}d\sigma =o_R(1),
$$
and
$$
\int_{B_R}(|x|^{2\alpha}e^{2u}x\cdot\nabla (V^2(x))-|x|^{\alpha}e^{u}|\Psi|^2x\cdot \nabla
V(x))dx=o_R(1),
$$
and
$$
\int_{\partial
B_R}\la\frac{\partial \Psi}{\partial \nu},x\cdot \nabla \Psi\ra d\sigma+\int_{\partial
B_R}\la x\cdot \nabla \Psi,\frac{\partial \Psi}{\partial \nu}\ra d\sigma=o_R(1).
$$
Putting all together and letting $R\rightarrow 0$, we
get
$$
C(u,\Psi)=\lim_{R\rightarrow 0}C(u,\Psi, R)=\frac {\gamma^2}{4\pi}.
$$
Since $C(u,\Psi)=0$ for $(u,\Psi)$, therefore we get $\gamma =0$. Then from the proof of Proposition \ref{prop-a} we have $(u,\Psi)\in C^2(B_r)\times C^2(\Gamma(\Sigma B_r))$, i.e. the local singularity of $(u,\Psi)$ is removable.
\end{proof}

\section{Bubble Energy}\label{bubble}

In this section, we shall analyze some properties of a  ``bubble'', i.e., an entire solution of
(\ref{Eq-L}) with finite energy and with constant coefficient function, which can be obtained after a suitable
rescaling at a blow-up point. We shall obtain the asymptotic behavior of an entire solution with finite
energy and show the global singularity removability. The latter means that  an entire solution on $\R^2$ can be conformally
extended to $\S^2$.

Without loss of generality, we assume that $V(x)\equiv1$ and hence the considered equations are
\begin{equation}\label{se}
\left\{
\begin{array}{rcl}
-\Delta u &=& \ds\vs 2|x|^{2\alpha} e^{2u}-|x|^{\alpha}e^u|\Psi|^2,\\
\slashiii{D}\Psi &=&\ds  -|x|^{\alpha} e^u\Psi,
\end{array}\qquad \text { in } \R^2.
\right.
\end{equation}
with $\alpha\geq 0$. The energy condition is
\begin{equation}\label{sec}
I(u,\Psi)=\int_{\R^2}(|x|^{2\alpha}e^{2u}+|\Psi|^4)dx<\infty.
\end{equation}

\ \

First, let $(u,\Psi)\in H^{1,2}_{loc}(\R^2) \times W^{1,\frac
43}_{loc}(\Gamma (\Sigma \R^2))$ be a weak solution of (\ref{se}) and (\ref{sec}), then applying similar arguments as in the proof of Proposition \ref{prop-a}, we get $u^{+}\in L^{\infty}(\R^2)$ and hence $(u,\Psi) \in C^2(\R^2)\times C^2(\Gamma(\Sigma \R^2))$.

Next, we denote by $(v,\Phi)$ the Kelvin transformation of $(u,\Psi)$,
i.e.
\begin{eqnarray*}
& & v(x)=u(\frac {x}{|x|^2})-(2+2\alpha)\ln |x|,\\
& & \Phi (x)=|x|^{-1} \Psi (\frac{x}{|x|^2})
\end{eqnarray*}
Then $(v,\Phi)$ satisfies
\begin{equation}\label{sek}
\left\{
\begin{array}{rcl}
-\Delta v &=& 2|x|^{2\alpha}e^{2v}-|x|^{\alpha}e^v|\Phi|^2,\\
\slashiii{D}\Phi &=&\ds  -|x|^{\alpha}e^v\Phi,
\end{array}\qquad x\in
\R^2\backslash\{0\}.
\right.
\end{equation}

\
\

Now we define the energy of the entire solution, i.e. the bubble energy, by $$ d =\int_{\R^2}2|x|^{2\alpha}e^{2u}
-|x|^{\alpha}e^u|\Psi|^2dx,$$ and define a constant spinor $\xi_0=\int_{\R^2}|x|^{\alpha}e^u\Psi dx$. It will  turn out that the
constant spinor $\xi_0$ is well defined. Then we have

\begin{prop}\label{asy}
Let $(u,\Psi)$ be a solution of (\ref{se}) and (\ref{sec}). Then
$u$ satisfies
\begin{equation}\label{ayu}
u(x)=-\frac d{2\pi} \ln{|x|}+C+O(|x|^{-1}) \qquad
\text{for}\quad |x| \quad \text{near}\quad \infty,
\end{equation}

\begin{equation}\label{aypsi}
\Psi (x)=-\frac {1}{2\pi}\frac{x}{|x|^2}\cdot
\xi_0+o(|x|^{-1})\qquad \text{for}\quad |x| \quad \text{near}\quad
\infty,
\end{equation}
where $\cdot$ is the Clifford multiplication, $C\in \R$ is some
constant, and $d =4\pi(1+\alpha)$.
\end{prop}

\begin{proof} The proof of this proposition is standard, see \cite{JWZ, CL2, JWZ1} and the references therein.  The essential facts used in this case are the Pohozaev identity (Proposition \ref{prop-poho}) and the decay estimate for the spinor of \eqref{sek} (see Lemma \ref{asy-phi}). For readers' convenience, we sketch the proof here.

\
\

First, let us define
$$
w(x)=-\frac{1}{2\pi}\int_{\R^2}(\ln{|x-y|}-\ln{(|y|+1)})(2|y|^{2\alpha}e^{2u}-|y|^{\alpha}e^u|\Psi|^2)dy.
$$
Since
$$
\int_{\R^2}(2|x|^{2\alpha}e^{2u}-|x|^{\alpha}e^u|\Psi|^2)dx<C,
$$
it follows from the standard potential argument that
$$
\frac {u(x)}{\ln{|x|}}\rightarrow -\frac {d}{2\pi}\qquad
\text{as} \quad |x|\rightarrow +\infty.
$$
Since $\int_{\R^2}|x|^{2\alpha}e^{2u}dx<+\infty$, the above result implies
$$
d \geq 2\pi(1+\alpha).
$$
Furthermore, similarly as in the case of the usual Liouville  or super-Liouville equation \cite{JWZ}, we can show that $d >2\pi(1+\alpha)$.

\
\

Secondly, from $d >2\pi(1+\alpha)$, we can improve the estimate for $e^{2u}$ to
\begin{equation}\label{ayu3}
e^{2u}\leq C|x|^{-2-2\alpha-\varepsilon}\qquad \text{for}\quad |x| \quad
\text{near}\quad \infty.
\end{equation}
Therefore, from Lemma \ref{asy-phi} and the Kelvin transformation, we obtain
the following asymptotic estimates of the spinor $\Psi(x)$:
\begin{equation}\label{asy-psi1}
|\Psi (x)|\leq C|x|^{-\frac 12-\delta_0}\qquad \text{for}\quad |x|
\quad \text{near}\quad \infty,
\end{equation}
and
\begin{equation}\label{asy-psi2}
|\nabla\Psi (x)|\leq C|x|^{-\frac 32-\delta_0}\qquad \text{for}\quad |x|
\quad \text{near}\quad \infty,
\end{equation}
for some positive number $\delta_0$.

Then, from (\ref{ayu3}), (\ref{asy-psi1}) and (\ref{asy-psi2}) and by some standard potential analysis in \cite{CL2} and \cite{CK}, we can obtain firstly
$$
-\frac{d}{2\pi}\ln{|x|}-C\leq u(x)\leq
-\frac{d}{2\pi}\ln{|x|}+C
$$
and furthermore we can get
\begin{equation*}
u(x)=-\frac d{2\pi} \ln{|x|}+C+O(|x|^{-1}) \qquad
\text{for}\quad |x| \quad \text{near}\quad \infty,
\end{equation*}
for some constant $C>0$. Thus we get the proof of (\ref{ayu}).

\
\

Next, we want to show that $d=4\pi(1+\alpha)$.
For sufficiently large $R>0$, the Pohozaev identity for the solution $(u,\Psi)$ gives
\begin{eqnarray}\label{222}
&&R\int_{\partial B_R(0)} |\frac {\partial u}{\partial \nu}|^2-\frac 12|\nabla u|^2d\sigma \nonumber\\
&=&(1+\alpha)\int_{B_R(0)}(2|x|^{2\alpha}e^{2u}-|x|^{\alpha}e^u|\Psi|^2)dx \nonumber\\
& & -R\int_{\partial B_R(0)}|x|^{2\alpha}e^{2u}d\sigma+\frac 12\int_{\partial
B_R(0)}\la\frac {\partial \Psi}{\partial \nu}, x\cdot\Psi\ra +\la
x\cdot\Psi, \frac {\partial \Psi}{\partial \nu}\ra d\sigma
\end{eqnarray}
where $\nu$ is the outward normal vector to $\partial B_R(0)$. By (\ref{ayu}), (\ref{asy-psi1}) and (\ref{asy-psi2}) we have
$$
\lim_{R\rightarrow +\infty}R\int_{\partial B_R(0)}
|\frac {\partial u}{\partial \nu}|^2-\frac 12|\nabla
u|^2d\sigma=\frac 1{4\pi}d^2,
$$
and
$$
\lim_{R\rightarrow +\infty} R\int_{\partial B_R(0)}|x|^{2\alpha}e^{2u}d\sigma=0,
$$
and
$$
\lim_{R\rightarrow +\infty}\int_{\partial
B_R}|\frac {\partial \Psi}{\partial \nu}||x\cdot \Psi|d\sigma=0.
$$
Let $R\rightarrow \infty$  in (\ref{222}), we
get that $$ \frac{1}{4\pi}d^2=(1+\alpha)d.$$ It follows that
$d=4\pi(1+\alpha)$.

\
\

Finally, we show (\ref{aypsi}).  Noting that $d =4\pi(1+\alpha)$,
we have
\begin{equation}\label{ayu6}
e^{2u}\leq C|x|^{-(4+4\alpha)} \qquad \text{for} \quad |x|\quad
\text{near}\quad \infty.
\end{equation}
This implies that the constant spinor $\xi_0$ is well defined.
By using the Green function of the Dirac operator in $\R^2$,
$$
G(x,y)=\frac {1}{2\pi}\frac {x-y}{|x-y|^2}\cdot,
$$
see \cite{AHM}, if we set
$$
\xi(x)=-\frac{1}{2\pi}\int_{\R^2} \frac {x-y}{|x-y|^2}\cdot
|x|^{\alpha}e^u\Psi dy,
$$
then we have
$\slashiii{D} \xi=-|x|^{\alpha}e^u\Psi$.

Since
\begin{eqnarray}\label{aypsi1}
 |x\cdot \xi(x)-\frac {1}{2\pi}\xi_0| \nonumber
&=& \frac{1}{2\pi}|\int_{\R^2}(\frac{x\cdot(x-y)}{|x-y|^2}+1)\cdot
|y|^{\alpha}e^u\Psi(y)dy| \nonumber\\
&=& \frac{1}{2\pi}|\int_{\R^2}\frac{(x-y)\cdot y}{|x-y|^2}\cdot
|y|^{\alpha}e^u\Psi(y)dy| \nonumber\\
&\leq & \frac{1}{2\pi} \int_{\R^2}\frac{|y|}{|x-y|}|y|^{\alpha}e^u|\Psi|dy,
\end{eqnarray}
and by (\ref{ayu6}),
\begin{equation}
|x|^{\alpha}|\Psi|e^u\leq C|x|^{-2-\alpha -\varepsilon} \qquad \text{for} \quad
|x|\quad \text{near}\quad \infty,
\end{equation}
for some positive constants $C$ and $\varepsilon$, we can follow
the derivation of gradient estimates in \cite{CK} to get
\begin{equation}\label{aypsi3}
|x\cdot \xi(x)-\frac{1}{2\pi}\xi_0|\leq C|x|^{-\varepsilon} \qquad
\text{for} \quad |x|\quad \text{near}\quad \infty.
\end{equation}

Set $\eta (x)=\Psi(x)-\xi(x)$. Then $\slashiii{D}\eta(x)=0$. It follows from (\ref{asy-psi1}) and (\ref{aypsi3}) that $\eta(x)=0$, i.e. $\Psi(x)=\xi(x)$. Furthermore,
\begin{eqnarray*}
|\Psi(x)+\frac{1}{2\pi}\frac{x}{|x|^2}\cdot \xi_0|
&=& |\frac{x}{|x|^2}\cdot(x\cdot \Psi(x)-\frac{1}{2\pi}\xi_0)|\\
&\leq & \frac{1}{|x|}|x\cdot\Psi(x)-\frac{1}{2\pi}\xi_0|\\
&\leq & C|x|^{-1-\varepsilon},
\end{eqnarray*}
for $|x|$ near $\infty$. This proves (\ref{aypsi}).
\end{proof}

\

Finally, we show that an entire solution can be conformally extended to $\S^2$.

\begin{thm}(Removability of a global singularity)\label{remove-gs}
Let $(u,\Psi)$ be a $C^2(\R^2)\times C^2(\Gamma(\Sigma \R^2))$ solution of (\ref{se}) and (\ref{sec}).
Then $(u,\Psi)$ extends conformally to a solution on $\S^2$.
\end{thm}

\begin{proof}
Let $(v,\Phi)$ be the Kelvin transformation of $(u,\Psi)$. Then
$(v,\Phi)$ satisfies (\ref{sek}) on $\R^2\backslash \{0\}$. To
prove this theorem, it is sufficient to show that $(v,\Phi)$ is smooth
on $\R^2$. Applying Proposition \ref{asy}, we have
\begin{equation}
v(x)=(\frac {d}{2\pi}-(2+2\alpha)) \ln{|x|}+O(1) \qquad \text{for}\quad
|x| \quad \text{near}\quad 0.
\end{equation}
Since $d =4\pi(1+\alpha)$, we get that $v$ is bounded near $0$. By
recalling that $\Phi$ is also bounded near $0$, standard elliptic
theory implies that $(v,\Phi)\in  C^2(\R^2)\times C^2(\Gamma(\Sigma \R^2))$.
\end{proof}

\
\

\section{Energy Identity for Spinors}\label{ener}

In this section, which is the technically most demanding one, we shall show an energy identity for the spinors. Firstly, analogously to the case of super-Liouville equations on closed Riemann surfaces
(see Lemma 3.4, \cite{JWZZ1}),  we shall derive the following local estimate for the spinor part on an annulus:

\begin{lm}\label{main-lamm}
Let $(u,\Psi)$ be a solution of (\ref{Eq-L}) and (\ref{Co-L}) on the annulus
$A_{r_1,r_2}=\{x\in \R^2|r_1\leq |x|\leq r_2\}$, where
$0<r_1<2r_1<\frac {r_2}2<r_2<1$. Then we have
\begin{eqnarray}\label{inqu}
&&(\int_{A_{2r_1,\frac {r_2}2}}|\nabla \Psi|^{\frac 43})^{\frac
34}+(\int_{A_{2r_1,\frac {r_2}2}}|\Psi|^4)^{\frac 14}\\
&\leq & \Lambda (\int_{A_{r_1,r_2}}|x|^{2\alpha}e^{2u})^{\frac
12}(\int_{A_{r_1,r_2}}|\Psi|^4)^{\frac
14}+C(\int_{A_{r_1,2r_1}}|\Psi|^4)^{\frac 14}+C(\int_{A_{\frac
{r_2}2, r_2}}|\Psi|^4)^{\frac 14}\nonumber
\end{eqnarray} for a positive constant $\Lambda$ and some universal
positive constant $C$.
\end{lm}

\begin{proof} In view of the second equation in (\ref{Eq-L}), one can apply the $L^p$ estimates for the Dirac operator $\slashiii
{D}$ and use similar arguments as in the proof of Lemma 3.4  of \cite{JWZZ1} to prove the lemma.
\end{proof}

Then, we can show the energy identity for the spinors - Theorem \ref{engy-indt}.

\

\noindent{\bf Proof of Theorem \ref{engy-indt}:}
We shall follow closely the arguments for the case of super-Liouville equations on closed Riemann surfaces \cite{JWZZ1}.
One crucial step here is to use the local singularity removability to get a contradiction.

We assume that $D_{\delta_i}$ be a small ball which is centered at a blow-up point
$x_i\in \Sigma_1$ such that $D_{2\delta_i}\bigcap D_{2\delta_j}=\emptyset$ for
$i\neq j, i,j=1,2,\cdots, l$, and on $B_r(0)\backslash
\bigcup_{i=1}^{l}D_{\delta_i}$, $\Psi_n$  converges strongly to
some limit $ \Psi$ in $L^4$ and $\int_{B_r(0)}|\Psi|^4 < \infty$. Then, it suffices to prove that for each fixed blow-up point $x_i \in \Sigma_1$, there are solutions $(u^{k},\xi^{k})$ of (\ref{Eq-L}) and (\ref{Co-L})
on $S^2$ with $\alpha\geq 0$ and $V$ being a constant function, $k=1,2,\cdots, K$
 such that
\begin{equation*}
\lim_{\delta_i\rightarrow 0}\lim_{n\rightarrow
\infty}\int_{D_{\delta_i}}|\Psi_n|^4dx=\sum_{k=1}^{K}\int_{S^2}|\xi^{k}|^4dx.
\end{equation*}

Without loss of generality, we assume that there is only one bubble at each blow-up point $p$ (the general case of multiple bubbles at $p$ can be handled by induction). Furthermore, we may assume that $p=0$. The case of $p\neq 0$ can be handled in an analogous way and in fact this case is simpler, as $|x|^{2\alpha_n}$ is a smooth function near $p\neq 0$. Then what we need to prove is that there exists a bubble $(u,\xi)$ such that

\begin{equation}\label{e2}
\lim_{\delta\rightarrow 0}\lim_{n\rightarrow
\infty}\int_{D_{\delta}}|\Psi_n|^4dx=\int_{S^2}|\xi|^4dx.
\end{equation}
where $D_\delta$ is a disc of radius $\delta>0$ centered at the blow-up point $p=0$.

We rescale each $(u_n,\Psi_n)$ near the blow-up point $p$. Choose $x_n \in \overline{D}_{\delta}$ such that
$u_n(x_n)=\max_{\overline{D}_{\delta}}u_n(x)$. Then we have $x_n\rightarrow p=0$ and
$u_n(x_n)\rightarrow +\infty$. Let $\lambda_n =e^{\frac{-u_n(x_n)}{\alpha_n+1}}\rightarrow 0$ and define $t_n=\max\{\lambda_n, |x_n|\}\rightarrow 0$.
Now there are two cases:
(i) $\frac{t_n}{\lambda_n}=O(1) \text{ as } n\rightarrow +\infty$ and (ii) $\frac{t_n}{\lambda_n}\rightarrow +\infty$ as $n\rightarrow +\infty.$

\
\

\noindent{\bf Case I:}
$\frac{t_n}{\lambda_n}=O(1)  \text{ as }  n\rightarrow +\infty.
$

\
\

In this case, we define
\begin{equation*}
\left\{
\begin{array}{rcl}
\widetilde{u}_n(x)&=&u_n(t_nx)+(\alpha_n+1)\ln {t_n}\\
\widetilde{\Psi}_n(x)&=&t_n^{\frac
12}\Psi_n(t_nx)
\end{array}
\right.
\end{equation*}
for any $x\in \overline{D}_{\frac{\delta}{2t_n}}$. Then $(\widetilde{u}_n(x),\widetilde{\Psi}_n(x))$ satisfies
\begin{equation*}
\left\{
\begin{array}{rcl}
-\Delta \widetilde{u}_n(x)&=& 2V^2(t_n x)|x|^{2\alpha_n}e^{2\widetilde{u}_n(x)}-V(t_n x)|x|^{\alpha_n}
e^{\widetilde{u}_n(x)}|\widetilde{\Psi}_n(x)|^2, \\
\slashiii
{D}\widetilde{\Psi}_n(x)&=&-V(t_n x)|x|^{\alpha_n}e^{\widetilde{u}_n(x)}\widetilde{\Psi}_n(x),
\end{array} \text{ in } \overline{D}_{\frac{\delta}{2t_n}}
\right.
\end{equation*}
with energy conditions
$$
\int_{D_{\frac{\delta}{2t_n}}}\left (|x|^{2\alpha_n}e^{2\widetilde{u}_n(x)}+|\widetilde{\Psi}_n(x)|^4\right )dx<C.
$$
Notice that
$$ 0\leq\max_{\overline{D}_{\frac{\delta}{2t_n}}}\widetilde{u}_n(x)=
\widetilde{u}_n(\frac{x_n}{t_n})=u_n(x_n)+(\alpha_n+1)\ln t_n=-(\alpha_n+1)\ln \lambda_n+(\alpha_n+1)\ln t_n\leq C.$$ Moreover, since
 the maximum point of $\widetilde {u}_n(x)$, i.e. $\frac{x_n}{t_n}$, is bounded, namely $|\frac{x_n}{t_n}|\leq 1$. So by taking a subsequence, we can assume that
$\frac{x_n}{t_n}\rightarrow x_0\in \R^2$  with $|x_0|\leq 1$. Therefore it follows from Theorem \ref{mainthm} that, by passing to a subsequence,  $(\widetilde{u}_n,\widetilde{\Psi}_n)$ converges in  $C^2_{loc}(\R^2)\times C^2_{loc}(\Gamma(\Sigma \R^2))$ to some $(\widetilde
u,\widetilde \Psi)$ satisfying
\begin{equation}
\left\{
\begin{array}{rcl}
-\Delta \widetilde{u} &=& 2V^2(0)|x|^{2\alpha}e^{2\widetilde{u}}-V(0)|x|^{\alpha}e^{\widetilde u}|
\widetilde \Psi|^2, \\
\slashiii{D}\widetilde{\Psi} &=&
-V(0)|x|^{\alpha}e^{\widetilde{u}}\widetilde{\Psi},  \end{array} \text { in } \R^2
\right.   \label{eq-13-1}
\end{equation}
with the energy condition
$\int_{\R^2}(|x|^{2\alpha}e^{2\widetilde{u}}+|\widetilde{\Psi}|^4)dx < \infty$.  By Proposition \ref{asy}, there holds
\[\int_{\R^2}(2V^2(0)|x|^{2\alpha}e^{2\widetilde{u}}-V(0)|x|^{\alpha}e^{\widetilde {u}}|\widetilde{\Psi}|^2)dx=4\pi(1+\alpha).\]
and by the removability of a global singularity (Theorem \ref{remove-gs}), we get a bubbling solution of \eqref{Eq-L} and \eqref{Co-L} on $S^2$.

\
\

{\bf Case II}:  $\frac{t_n}{\lambda_n}\rightarrow +\infty$ as $n\rightarrow +\infty.$

In this case, necessarily $t_n=|x_n|$ and hence $ \frac{|x_n|}{\lambda_n}\rightarrow +\infty$ as $n\rightarrow +\infty$. Set
$\tau_n=\frac{e^{-u_n(x_n)}}{|x_n|^{\alpha_n}}=\lambda_n(\frac{\lambda_n}{|x_n|})^{\alpha_n}$. Then $\tau_n\rightarrow 0$ and
$\frac{|x_n|}{\tau_n}\rightarrow +\infty$, as $n\rightarrow +\infty$. Now define
\begin{equation*}
\left\{
\begin{array}{rcl}
\widetilde{u}_n(x)&=&u_n(x_n+\tau_nx)-u_n(x_n)\\
\widetilde{\Psi}_n(x)&=&\tau_n^{\frac
12}\Psi_n(x_n+\tau_nx)
\end{array}
\right.
\end{equation*}
for any  $x\in \overline{D}_{\frac{t_n \delta}{2\tau_n}}(x_n)$. Then $(\widetilde{u}_n(x),\widetilde{\Psi}_n(x))$ satisfies
\begin{equation*}
\left\{
\begin{array}{rcl}
-\Delta \widetilde{u}_n(x)&=& 2V^2(x_n+\tau_nx)|\frac{x_n}{|x_n|}+\frac{\tau_n}{|x_n|}x|^{2\alpha_n}e^{2\widetilde{u}_n(x)}\\
& & -V(x_n+\tau_nx)|\frac{x_n}{|x_n|}+\frac{\tau_n}{|x_n|}x|^{\alpha_n}
e^{\widetilde{u}_n(x)}|\widetilde{\Psi}_n(x)|^2, \\
\slashiii
{D}\widetilde{\Psi}_n(x)&=&-V(x_n+\tau_nx)|\frac{x_n}{|x_n|}+\frac{\tau_n}{|x_n|}x|^{\alpha_n}e^{\widetilde{u}_n(x)}\widetilde{\Psi}_n(x),
\end{array}
\right.
\end{equation*}
in $D_{\frac{t_n \delta}{2\tau_n}}(x_n)$ and with energy conditions
$$
\int_{D_{\frac{t_n \delta}{2\tau_n}}}\left(|\frac{x_n}{|x_n|}+\frac{\tau_n}{|x_n|}x|^{2\alpha_n}e^{2\widetilde{u}_n(x)}+|\widetilde{\Psi}_n(x)|^4\right)dx<C.
$$
It is clear that $\widetilde{u}_n(x)\leq \max_{\overline{D}_{\frac{t_n\delta}{2\tau_n}}(x_n)}\widetilde{u}_n(x)=\widetilde{u}_n(0)=0$, and
 $|\frac{x_n}{|x_n|}+\frac{\tau_n}{|x_n|}x|^{2\alpha_n}\rightarrow 1$ uniformly in $C^0_{loc}(\R^2)$. Then from Theorem \ref{mainthm},
by passing to a subsequence,  $(\widetilde{u}_n,\widetilde{\Psi}_n)$ converges in  $C^2_{loc}(\R^2)\times C^2_{loc}(\Gamma(\Sigma \R^2))$ to some $(\widetilde
u,\widetilde \Psi)$ satisfying
\begin{equation*}
\left\{
\begin{array}{rcl}
-\Delta \widetilde{u} &=& 2V^2(0)e^{2\widetilde{u}}-V(0)e^{\widetilde u}|
\widetilde{\Psi}|^2, \\
\slashiii{D}\widetilde{\Psi} &=&
-V(0)e^{\widetilde{u}}\widetilde{\Psi},  \end{array} \text { in } \R^2
\right.
\end{equation*}
with the energy condition
$\int_{\R^2}(e^{2\widetilde{u}}+|\widetilde{\Psi}|^4)dx < \infty$. By the removability of a global singularity (see Proposition 6.3 and Theorem 6.4 in \cite{JWZ}), there holds
\[\int_{\R^2}(2V^2(0)e^{2\widetilde{u}}-V(0)e^{\widetilde{u}}|\widetilde{\Psi}|^2)dx=4\pi\]
and we get a bubbling solution of \eqref{Eq-L} and \eqref{Co-L} on $S^2$.

\

In order to prove (\ref{e2}) we need to estimate the energy of $\Psi_n$ in the neck domain.
We shall proceed separately for Case I and for Case II.

\
\

For {\bf Case I}, the neck domain is
$$ A_{\delta, R, n}=\{x\in \R^2|t_n R\leq |x|\leq
\delta\}.$$  Then to prove (\ref{e2}), it suffices to prove the following
\begin{equation}\label{e3}
\lim_{\delta\rightarrow 0}\lim_{R\rightarrow +\infty}\lim_{n\rightarrow
\infty}\int_{A_{\delta,R,n}}|\Psi_n|^4dx=0.
\end{equation}

Next we shall show two claims.

\
\

\noindent {\bf Claim I.1}: For any
$\epsilon>0$, there is an $N>1$ such that for any $n\geq N$, we
have
$$
\int_{D_r\setminus
D_{e^{-1}r}}(|x|^{2\alpha_n}e^{2u_n}+|\Psi_n|^4)dx<\epsilon, \quad
\forall r\in [et_n R, \delta].
$$

\

To show this claim, we firstly note the following two facts:

\

\noindent{\bf Fact I.1}: For any $\epsilon>0$ and any $T>0$, there exists some $N(T)>0$ such that for any $n\geq N(T)$, we
have
\begin{equation}\label{2.6}
\int_{D_\delta \setminus D_{\delta
e^{-T}}}(|x|^{2\alpha_n}e^{2u_n}+|\Psi_n|^4)dx<\epsilon.
\end{equation}
Actually, since $(u_n,\Psi_n)$ has no blow-up point in $\overline{D}_{2\delta}\backslash \{p\}$, we know that $\Psi_n$ converges strongly to $\Psi$ in
$L_{loc}^4(\overline{D}_{2\delta}\backslash \{p\})$,  and $u_n$ will either be uniformly bounded on any compact subset of
$\overline{D}_{2\delta}\backslash \{p\}$ or uniformly tend to $-\infty$ on any compact subset of $\overline{D}_{2\delta}\backslash \{p\}$.

If $u_n$ uniformly tends to $-\infty$ on any compact subset of $\overline{D}_{2\delta}\backslash \{p\}$, it is clear that,
for any given $T>0$, there is an $N(T)>0$ big enough such that when $n\geq N(T)$, we have

\begin{equation*}  \int_{D_\delta\backslash D_{\delta
e^{-T}}}|x|^{2\alpha_n}e^{2u_n} dx< \frac{\epsilon}{2}.
\end{equation*}

Moreover, since $\Psi_n$ converges to $\Psi$ in
$L_{loc}^4(\overline{D}_{2\delta}\backslash \{p\})$ and hence
$$
\int_{D_\delta\backslash D_{\delta
e^{-T}}}|\Psi_n|^{4} \rightarrow \int_{D_\delta\backslash
D_{\delta e^{-T}}}|\Psi|^{4}.
$$
For any given $\epsilon>0$ small, we can choose $\delta>0$ small enough
such that  $\int_{D_\delta}|\Psi|^{4}<\frac \epsilon 4$, then
for any given $T>0$, there is an $N(T)>0$ big enough such that when $n\geq N(T)$
$$
\int_{D_\delta\backslash D_{\delta
e^{-T}}}|\Psi_n|^{4}<\frac{\epsilon}{2}.
$$
Consequently, we get (\ref{2.6}).

\

If $(u_n, \Psi_n)$ is uniformly bounded on any compact subset of $\overline{D}_{2\delta}\backslash \{p\}$, then
$(u_n, \Psi_n)$ converges to a limit solution $(u,\Psi)$ with bounded energy
$\int_{D_{2\delta}}(|x|^{2\alpha}e^{2u}+|\Psi|^4)< \infty$  strongly on any compact subset of $\overline{D}_{2\delta}\backslash \{p\}$
and hence
\begin{equation*}  \int_{D_\delta\backslash D_{\delta
e^{-T}}}(|x|^{2\alpha_n}e^{2u_n}+|\Psi_n|^4) \rightarrow   \int_{D_{\delta}\backslash D_{\delta e^{-T}}}(|x|^{2\alpha}e^{2u}+|\Psi|^4)
\end{equation*}

Therefore, we can choose $\delta>0$ small enough such that, for any given $\epsilon>0$  and any given $T>0$,
there exists an $N(T)>0$ big enough so that, when $n\geq N(T)$, (\ref{2.6}) holds.

\

\noindent{\bf Fact I.2}: For any small $\epsilon>0$, and $T>0$, we may choose an $N(T)>0$ such that when $n\geq N(T)$

\begin{eqnarray*}
& & \int_{D_{t_nRe^T}\setminus
D_{t_nR}}(|x|^{2\alpha_n}e^{2u_n}+|\Psi_n|^4)\\
&=& \int_{D_{Re^T}\setminus
D_{R}}(|x|^{2\alpha_n}e^{2\widetilde{u}_n}+|\widetilde{\Psi}_n|^4)\\
&\rightarrow&  \int_{D_{Re^T}\setminus
D_{R}}(|x|^{2\alpha}e^{2\widetilde{u}}+|\widetilde{\Psi}|^4)\\
& < & \epsilon,
\end{eqnarray*}
if $R$ is big enough.

\

Now we can deal with {\bf Claim I.1}. We argue  by contradiction by using the
above two facts.  Suppose that there exists $\epsilon_0>0$ and a sequence $r_n\in [et_nR,\delta]$ such that
$$
\int_{D_{r_n}\setminus
D_{e^{-1}r_n}}(|x|^{2\alpha_n}e^{2u_n}+|\Psi_n|^{4})\geq \epsilon_0.
$$
Then, by the above two facts, we know that $\frac
{\delta}{r_n}\rightarrow +\infty$ and
$\frac{t_nR}{r_n}\rightarrow 0$, in particular, $r_n\rightarrow
0$  and $\frac{t_n}{r_n}\rightarrow 0 $ as $n\rightarrow +\infty$.

Scaling again, we set
\begin{equation}\label{scal-1}
\left\{
\begin{array}{rcl}
v_n(x)&=& u_n(r_nx)+(\alpha_n+1)\ln r_n, \\
\varphi_n(x)&=& r_n^{\frac 12}\Psi_n(r_nx).\\
\end{array}
\right.
\end{equation}

It is clear that
\begin{equation}\label{3.1}
\int_{(D_1\setminus D_{e^{-1}})}(|x|^{2\alpha_n}e^{2v_n}+|\varphi_n|^4)\geq \epsilon_0,
\end{equation}
and  $(v_n,\varphi_n)$ satisfies
\begin{equation*}
\left\{
\begin{array}{rcl}
-\Delta v_n(x) &=& 2V^2(r_nx)|x|^{2\alpha_n}e^{2v_n(x)}-V(r_nx)|x|^{\alpha_n}e^{v_n(x)}
|\varphi_n(x)|^2, \\
\slashiii{D}\varphi_n(x) &=& -V(r_nx)|x|^{\alpha_n}e^{v_n(x)}\varphi_n(x), \\
\end{array}
\right.
\end{equation*}
in $D_{\frac{\delta}{r_n}}\setminus D_{\frac{t_nR}{r_n}}$.
By Theorem \ref{mainthm}, there are three possible cases:

\

 (1). There exists some $R>0$, some point $q \in D_R\setminus D_{\frac 1R}$ and energy concentration occurs near $q$, namely along some subsequence
 $$\lim_{n\rightarrow \infty} \int_{D_r(q)}(|x|^{2\alpha_n}e^{2v_n}+|\varphi_n|^{4})\geq \epsilon_0>0$$ for any small $r>0$. In such
a case, we still obtain a second bubble on $S^2$ by the rescaling argument. Thus we get a contradiction
to the assumption that there is only one bubble at the blow-up point $p$.

\

(2). For any $R>0$, there is no blow-up point in $D_R\setminus D_{\frac 1R}$ and
$v_n$ tends to $-\infty$ uniformly in $\overline {D_R\setminus D_{\frac 1R}}$.
Then, there is a solution $\varphi$ satisfying
\begin{equation*}
\slashiii{D}\varphi = 0,  \text { in } \R^2\setminus\{0\},
\end{equation*}
with bounded  energy $||\varphi||_{L^4(\R^2)} < \infty$, such that
$$\lim_{n\rightarrow \infty}   ||\varphi_n - \varphi||_{L^{4}(D_R\setminus D_{\frac 1R})}=0, \quad {\rm for \ any}\ R>0.$$
By the same arguments as in the case of super-Liouville equations \cite{JWZZ1}, we know that $\varphi$ can be conformally
extended to a harmonic spinor on $S^2$, which has to be identically $0$. This will contradict (\ref{3.1}).

\

(3). For any $R>0$, there is no blow-up point in $D_R\setminus D_{\frac 1R}$ and
$(v_n,\varphi_n)$ is uniformly bounded in $\overline {D_R\setminus D_{\frac 1R}}$.
Then, there is a solution $(v,\varphi)$ satisfying
\begin{equation}\label{eq-6}
\left\{
\begin{array}{rcll}
-\Delta v &=& 2V^2(0)|x|^{2\alpha}e^{2v}-V(0)|x|^{\alpha}e^{v}|\varphi|^2, & \quad \text{ in } \R^2\setminus\{0\} \\
\slashiii{D}\varphi &=& -V(0)|x|^{\alpha}e^{v}\varphi, & \quad \text{ in
}\R^2\setminus\{0\}\\
\end{array}
\right.
\end{equation}
with finite energy
$\int_{\R^2}(|x|^{2\alpha}e^{2v}+|\varphi|^4)dx < \infty$, such that
$$\lim_{n\rightarrow \infty}  \left (  ||v_n - v ||_{C^{2}(D_R\setminus D_{\frac 1R})} +
 ||\varphi_n - \varphi||_{C^{2}(D_R\setminus D_{\frac 1R})} \right ) =0,
$$ for any $R>0$.

\
\

In this case, we shall show that the local singularities at $0$ and at $\infty$ of $(v,\varphi)$ are removable.
Firstly, since $(u_n,\Psi_n)$ satisfies (\ref{Eq-Sn}) and \eqref{Eq-Cn} in $D_{2\delta}$, the following Pohozaev identity holds for any $\rho>0$ with $r_n\rho<2\delta$,
\begin{eqnarray*}
&&r_n\rho\int_{\partial D_{r_n\rho}} |\frac {\partial u_n}{\partial \nu}|^2-\frac 12|\nabla u_n|^2d\sigma \\
&=&(1+\alpha_n)\int_{D_{r_n\rho}}(2V^2(x)|x|^{2\alpha_n}e^{2u_n}-V(x)|x|^{\alpha_n}e^{u_n}|\Psi_n|^2)dx \\
& & -r_n\rho\int_{\partial D_{r_n\rho}}V^2(x)|x|^{2\alpha_n}e^{2u_n}d\sigma+\frac 12\int_{\partial
D_{r_n\rho}}\la\frac {\partial \Psi_n}{\partial \nu}, x\cdot\Psi_n\ra +\la
x\cdot\Psi_n, \frac {\partial \Psi_n}{\partial \nu}\ra d\sigma\\
& & +\int_{D_{r_n\rho}}(|x|^{2\alpha_n}e^{2u_n}x\cdot\nabla (V^2(x))-|x|^{\alpha_n}e^{u_n}|\Psi_n|^2x\cdot \nabla
V(x))dx.
\end{eqnarray*}
It follows that the associated Pohozaev constant of $(v_n(x), \varphi_n(x))$ (see \eqref{scal-1}) satisfies
\begin{eqnarray*}
 C(v_n,\varphi_n)  &=& C(v_n,\varphi_n,\rho) \\
&=& \rho\int_{\partial D_\rho} |\frac {\partial v_n}{\partial \nu}|^2-\frac 12|\nabla v_n|^2d\sigma \\
& & -(1+\alpha_n)\int_{D_\rho}(2V^2(r_nx)|x|^{2\alpha_n}e^{2v_n}-V(r_nx)|x|^{\alpha_n}e^{v_n}|\varphi_n|^2)dx \\
 & & +\rho\int_{\partial D_\rho}V^2(r_nx)|x|^{2\alpha_n}e^{2v_n}d\sigma-\frac 12\int_{\partial
D_\rho}\la\frac {\partial \varphi_n}{\partial \nu}, x\cdot\varphi_n\ra +\la
x\cdot\varphi_n, \frac {\partial \varphi_n}{\partial \nu}\ra d\sigma\\
& & -\int_{D_{\rho}}\left(|x|^{2\alpha_n}e^{2v_n}x\cdot\nabla (V^2(r_nx))-|x|^{\alpha_n}e^{v_n}|\varphi_n|^2x\cdot \nabla
(V(r_nx))\right)dx \\
&=& 0.
\end{eqnarray*}
It is easy to verify that
$$\lim_{\rho\rightarrow 0}\lim_{n\rightarrow \infty} \int_{D_{\rho}}\left(|x|^{2\alpha_n}e^{2v_n}x\cdot\nabla (V^2(r_nx))-|x|^{\alpha_n}e^{v_n}|\varphi_n|^2x\cdot \nabla
(V(r_nx))\right)dx=0.$$ Since $(v_n,\varphi_n)$ converges to $(v,\varphi)$ in $C^{2}_{loc}(\R^2\setminus \{0\})\times C^{2}_{loc}(\Gamma (\Sigma \R^2\setminus \{0\})) $, we have
\begin{eqnarray*}
0 &=& \lim_{\rho \rightarrow 0}\lim_{n\rightarrow \infty}C (v_n,\varphi_n, \rho)\\
 &=& \lim_{\rho \rightarrow 0}C(v,\varphi,\rho)-(1+\alpha)\lim_{\delta\rightarrow 0}\lim_{n\rightarrow \infty}\int_{D_\delta}(2V^2(r_nx)|x|^{2\alpha_n}e^{2v_n}-V(r_nx)|x|^{\alpha_n}e^{v_n}|\varphi_n|^2)dx\\
&=& C(v,\varphi)-(1+\alpha)\beta.
\end{eqnarray*}
Here $$\beta=\lim_{\delta\rightarrow 0}\lim_{n\rightarrow \infty}\int_{D_\delta}(2V^2(r_nx)|x|^{2\alpha_n}e^{2v_n}-V(r_nx)|x|^{\alpha_n}e^{v_n}|\varphi_n|^2)dx,$$
and $C(v,\varphi)= C(v,\varphi,\rho)$ is the Pohozaev constant with respect to the equation \eqref{eq-6}, i.e.
\begin{eqnarray*}
C(v,\varphi)= C(v,\varphi,\rho) &=& \rho\int_{\partial D_\rho} |\frac {\partial v}{\partial \nu}|^2-\frac 12|\nabla v|^2d\sigma \\
&& -(1+\alpha)\int_{D_\rho}(2V^2(0)|x|^{2\alpha}e^{2v}-V(0)|x|^{\alpha}e^{v}|\varphi|^2)dx \\
&& +\rho\int_{\partial D_\rho}V^2(0)|x|^{2\alpha}e^{2v}d\sigma-\frac 12\int_{\partial
D_\rho}\la\frac {\partial \varphi}{\partial \nu}, x\cdot\varphi\ra +\la
x\cdot\varphi, \frac {\partial \varphi}{\partial \nu}\ra d\sigma.
\end{eqnarray*}

\
\

On the other hand, since $(v_n,\varphi_n)$ converges to $(v,\varphi)$ in $C^{2}_{loc}(\R^2\setminus \{0\})\times C^{2}_{loc}(\Gamma (\Sigma \R^2\setminus \{0\})) $, we have
$$
2V^2(r_nx)|x|^{2\alpha_n}e^{2v_n}-V(r_nx)|x|^{\alpha_n}e^{v_n}|\varphi_n|^2\rightarrow \nu=2V^2(0)|x|^{2\alpha}e^{2v}-V(0)|x|^{\alpha}e^{v}|\varphi|^2+\beta\delta_{p=0}
$$  weakly in the sense of measures in $B_R$ for any small $R>0$. Using Green's representation formula for $(v_n,\varphi_n)$ in $B_R$, we derive that
$$
v(x)=-\frac{\beta}{2\pi}\log|x|+w(x)+h(x),
$$
with
$$
w(x)=-\frac 1{2\pi}\int_{B_R}(\log|x-y|)(2V^2(0)|y|^{2\alpha}e^{2v(y)}-V(0)|y|^{\alpha}e^{v(y)}|\varphi|^2(y))dy,
$$
and
$$
h(x)=\frac 1{2\pi}\int_{\partial B_R}(\log|x-y|)\frac{\partial v(y)}{\partial \nu}dy-\frac 1{2\pi}\int_{\partial B_R}\frac{(x-y)\cdot \nu}{|x-y|^2}v(y)dy.
$$
It is clear that $h(x)$ is a regular term and $h(x)\in C^1(B_R)$ and that $w(x)$ satisfies
$$
-\Delta(w(x)+h(x))=2V^2(0)|x|^{2\alpha}e^{2v(x)}-V(0)|x|^{\alpha}e^{v(x)}|\varphi|^2(x), \quad \text{ in } B_R.
$$
Therefore, applying similar arguments as in the proof of Proposition \ref{sigu-move1}, we know that $w(x)$ is bounded in $B_R$, and furthermore we obtain
$$
C(v,\varphi)=\frac {\beta^2}{4\pi}.
$$
Thus there holds
$$
\frac {\beta^2}{4\pi}=(1+\alpha)\beta.
$$
Since $\int_{B_R}|x|^{2\alpha}e^{2v}dx<\infty$, we have $\beta\leq 2\pi(1+\alpha)$. Therefore we conclude that $\beta=0$ and hence $C(v,\varphi)=0$. Then, by Proposition \ref{sigu-move1}, the singularity at 0 can be removed. Furthermore, the singularity at $\infty$ can be removed by applying the removability of a global singularity (see Theorem \ref{remove-gs}). Then we get another bubble on $\S^2$. Thus we get a contradiction and complete the proof of {\bf Claim I.1}.

\

\noindent {\bf Claim I.2}:  We can separate $A_{\delta, R, n}$ into
finitely many parts
$$
A_{\delta, R, n}=\bigcup_{k=1}^{N_k}A_k$$ such that on each part
\begin{equation}\label{e5}
\int_{A_k}|x|^{2\alpha_n}e^{2u_n}dx\leq \frac {1}{4\Lambda^2}, \quad k=1,2,\cdots, N_k.
\end{equation}
Where  $N_k\leq N_0$ with $N_0$ being an uniform integer for all
$n$ large enough, $A_k=D_{r^{k-1}}\setminus D_{r^k}$, $r^0=\delta, r^{N_k}=t_nR
$, $r^k<r^{k-1}$ for $k=1,2,\cdots, N_k$, and $\Lambda$ is the constant as in Lemma \ref{main-lamm}.

\

The proof of the above claim is standard, see the case of super-Liouville equations in \cite{JWZZ1, JZZ} as well as the cases of
other Dirac equations in \cite{Z1, Z2}. Here we omit it.

\

Now using {\bf Claim I.1} and {\bf Claim I.2}, we can show (\ref{e3}). The arguments are similar to the case of super-Liouville equations in \cite{JWZZ1, JZZ}. For the sake of completeness, we provide the details here.

Let $0<\epsilon<1$ be small, $\delta$ be small enough, and let $R$ and $ n$ be large enough. We apply Lemma \ref{main-lamm} to each
part $A_l$  and use \eqref{e5} to calculate
\begin{eqnarray*}
(\int_{A_l}|\Psi_n|^4)^{\frac 14} &\leq & \Lambda
(\int_{D_{er^{l-1}}\setminus D_{e^{-1}r^l}}|x|^{2\alpha_n}e^{2u_n})^{\frac
12}(\int_{D_{er^{l-1}}\setminus D_{e^{-1}r^l}}|\Psi_n|^4)^{\frac
14}\\
&&+C(\int_{D_{er^{l-1}}\setminus D_{r^{l-1}}}|\Psi_n|^4)^{\frac
14}+C(\int_{D_{r^{l}}\setminus D_{e^{-1}r^l}}|\Psi_n|^4)^{\frac 14}\\
&\leq & \Lambda ((\int_{A_l}|x|^{2\alpha_n}e^{2u_n})^{\frac 12}+\epsilon^{\frac
12}+\epsilon^{\frac 12})((\int_{A_l}|\Psi_n|^4)^{\frac
14}+\epsilon^{\frac 14}+\epsilon^{\frac 14})+C\epsilon^{\frac 14}\\
&\leq &\Lambda (\int_{A_l}|x|^{2\alpha_n}e^{2u_n})^{\frac
12}(\int_{A_l}|\Psi_n|^4)^{\frac 14}+C(\epsilon^{\frac
14}+\epsilon^{\frac 12}+\epsilon^{\frac 34})\\
& \leq & \frac 12 (\int_{A_l}|\Psi_n|^4)^{\frac
14}+C \epsilon^{\frac 14},
\end{eqnarray*}
which gives
\begin{equation}\label{2.1} (\int_{A_l}|\Psi_n|^4)^{\frac 14}\leq C \epsilon^{\frac 14}.
\end{equation}
Then, using Lemma \ref{main-lamm}, \eqref{e5}, (\ref{2.1}) and applying similar arguments, we obtain
\begin{equation}\label{2.2}
(\int_{A_l}|\nabla\Psi_n|^{\frac 43})^{\frac 34}\leq
C\epsilon^{\frac 14}.
\end{equation}
Summing up (\ref{2.1}) and (\ref{2.2}) on $A_l$, we conclude that
\begin{equation}\label{2.3}
\int_{A_{\delta,
R,n}}|\Psi_n|^4+\int_{A_{\delta,R,n}}|\nabla\Psi_n|^{\frac
43}=\sum_{l=1}^{N_0}\int_{A_l}|\Psi_n|^4+|\nabla\Psi_n|^{\frac
43}\leq C\epsilon^{\frac 13}.
\end{equation}
This proves (\ref{e3}) and finishes the proof of theorem in this case.

\

For {\bf Case II}, the neck domain is different from {\bf Case I} and it is
$$ A_{S, R, n}(x_n)=\{x\in \R^2|\tau_n R\leq |x-x_n|\leq t_n S\}.$$
In fact, in this case, we can rescale twice to get the bubble. First, since $t_n=|x_n|$, we define the rescaling functions

\begin{equation*}
\left\{
\begin{array}{rcl}
\bar{u}_n(x)&=&u_n(t_nx)+(\alpha_n+1)\ln {t_n}\\
\bar{\Psi}_n(x)&=&t_n^{\frac
12}\Psi_n(t_nx)
\end{array}
\right.
\end{equation*}
for any $x\in \overline{D}_{\frac{\delta}{2t_n}}$. Then $(\bar{u}_n(x),\bar{\Psi}_n(x))$ satisfies
\begin{equation*}
\left\{
\begin{array}{rcl}
-\Delta \bar{u}_n(x)&=& 2V^2(t_n x)|x|^{2\alpha_n}e^{2\bar{u}_n(x)}-V(t_n x)|x|^{\alpha_n}
e^{\bar{u}_n(x)}|\bar{\Psi}_n(x)|^2, \\
\slashiii
{D}\bar{\Psi}_n(x)&=&-V(t_n x)|x|^{\alpha_n}e^{\bar{u}_n(x)}\bar{\Psi}_n(x),
\end{array} \text{ in } \overline{D}_{\frac{\delta}{2t_n}}
\right.
\end{equation*}
with energy conditions
$$
\int_{D_{\frac{\delta}{2t_n}}}\left (|x|^{2\alpha_n}e^{2\bar{u}_n(x)}+|\bar{\Psi}_n(x)|^4\right )dx<C.
$$
Set that $y_n=\frac{x_n}{t_n}$. Noticing that $\bar{u}_n(y_n)=u_n(x_n)+(\alpha_n+1)\ln t_n=(\alpha_n+1)\ln t_n-(\alpha_n+1)\ln \lambda_n\rightarrow +\infty$, we set that $\delta_n=e^{-\bar{u}_n(y_n)}$ and define the rescaling function
\begin{equation*}
\left\{
\begin{array}{rcl}
\widetilde{u}_n(x)&=&\bar{u}_n(\delta_nx+y_n)+\ln {\delta_n}\\
\widetilde{\Psi}_n(x)&=&\delta_n^{\frac 12}\bar{\Psi}_n(\delta_nx+y_n)
\end{array}
\right.
\end{equation*}
for any $\delta_nx+y_n\in \overline{D}_{\frac{\delta}{2t_n}}$. We can see that $(\widetilde {u}_n, \widetilde{\Psi}_n)$ is exactly the same as that defined before. Without loss of generality, we assume that $y_0=\lim_{n\rightarrow \infty}\frac{x_n}{t_n}$.  Notice that
\begin{eqnarray*}
\int_{D_\delta}|\Psi_n|^4dx&=&\int_{D_\frac{\delta}{t_n}}|\bar{\Psi}_n|^4dx\\
&=&\int_{D_\frac{\delta}{t_n}\backslash D_{R_1}(y_n)}|\bar{\Psi}_n|^4dx+\int_{D_{R_1}(y_n)\backslash D_{\delta_n R_2}(y_n)}|\bar{\Psi}_n|^4dx+\int_{D_{\delta_n R_2}(y_n)}|\bar{\Psi}_n|^4dx\\
&=&\int_{D_\frac{\delta}{t_n}\backslash D_{R_1}(y_n)}|\bar{\Psi}_n|^4dx+\int_{D_{t_nR_1}(x_n)\backslash D_{t_n\delta_n R_2}(x_n)}|\Psi_n|^4dx+\int_{D_{\delta_n R_2}(y_n)}|\bar{\Psi}_n|^4dx.\\
\end{eqnarray*}
Since we have assumed that $(u_n,\Psi_n)$ has only one bubble at the blow-up point $p=0$, $(\bar{u}_n,\bar{\Psi}_n)$ also has only one bubble at the blow-up point $p=y_0$. Therefore we have  $$\lim_{R_1\rightarrow +\infty}\lim_{n\rightarrow \infty}\int_{D_\frac{\delta}{t_n}\backslash D_{R_1}(y_n)}|\bar{\Psi}_n|^4dx=0$$ uniformly for any small $\delta$, and since $D_{\delta_n R_2}(y_n)$ is a bubble domain, we know $A_{S,R,n}$ is the neck domain for sufficiently large $S,R>0$, and it is sufficient to prove
\begin{equation}\label{e3-xn}
\lim_{S\rightarrow +\infty}\lim_{R\rightarrow +\infty}\lim_{n\rightarrow
\infty}\int_{A_{S,R,n}(x_n)}|\Psi_n|^4dv=0.
\end{equation}

\
\

For this purpose, we shall prove two claims.

\
\

\noindent{\bf Claim II.1}:  for any $\epsilon>0$, there is an $N>1$ such that for any $n\geq N$, we
have
\begin{equation}\label{6.1}
\int_{D_r(x_n)\setminus
D_{e^{-1}r}(x_n)}(|x|^{2\alpha_n}e^{2u_n}+|\Psi_n|^4)dx<\epsilon, \quad
\forall r\in [e\tau_n R, t_nS].
\end{equation}

\
\

To get \eqref{6.1}, similarly to the {\bf Case I}, we firstly note the following two facts:

\
\

\noindent{\bf Fact II.1}: For any $\epsilon>0$ and any $T>0$, there exists some $N(T)>0$ such that for any $n\geq N(T)$, we
have
\begin{equation*}
\int_{D_{t_nS} (x_n)\setminus D_{t_nS
e^{-T}}(x_n)}(|x|^{2\alpha_n}e^{2u_n}+|\Psi_n|^4)dx<\epsilon.
\end{equation*} if $S$ is large enough.

\

{\bf Fact II.2}: For any small $\epsilon>0$, and $T>0$, we may choose an $N(T)>0$ such that when $n\geq N(T)$
\begin{eqnarray*}
& & \int_{D_{\tau_nRe^T}(x_n)\setminus
D_{\tau_nR}(x_n)}(|x|^{2\alpha_n}e^{2u_n}+|\Psi_n|^4)\\
&=& \int_{D_{Re^T}\setminus
D_{R}}(|\frac {x_n}{|x_n|}+\frac{\tau_n}{|x_n|}x|^{2\alpha_n}e^{2\widetilde{u}_n}+|\widetilde{\Psi}_n|^4)\\
&\rightarrow&  \int_{D_{Re^T}\setminus
D_{R}}(e^{2\widetilde{u}}+|\widetilde{\Psi}|^4)\\
& < & \epsilon,
\end{eqnarray*}
if $R$ is large enough.

\

Now we argue  by contradiction to show \eqref{6.1} by using the
above two facts. We assume that there exists $\epsilon_0>0$ and a sequence $r_n\in [e\tau_nR, t_n S]$ such that
$$
\int_{D_{r_n}(x_n)\setminus
D_{e^{-1}r_n}(x_n)}(|x|^{2\alpha_n}e^{2u_n}+|\Psi_n|^{4})\geq \epsilon_0.
$$
Then, by the above two facts, we know that $\frac {t_nS}{r_n}\rightarrow +\infty$ and
$\frac{\tau_nR}{r_n}\rightarrow 0$, in particular, $r_n\rightarrow 0$  as $n\rightarrow +\infty$. Note that $|\frac {x_n}{r_n}|= |\frac {t_n}{r_n}| \rightarrow +\infty$ as $n\rightarrow \infty$. We define
\begin{equation}\label{scal-2}
\left\{
\begin{array}{rcl}
v_n(x)&=& u_n(r_nx+x_n)+\ln (r_n|x_n|^{\alpha_n}), \\
\varphi_n(x)&=& r_n^{\frac 12}\Psi_n(r_nx+x_n).\\
\end{array}
\right.
\end{equation}
Then $(v_n,\varphi_n)$ satisfies
\begin{equation*}
\left\{
\begin{array}{rcl}
-\Delta v_n(x) &=& 2V^2(r_nx+x_n)|\frac{x_n}{|x_n|}+\frac{r_nx}{|x_n|}|^{2\alpha_n}e^{2v_n(x)}-V(r_nx+x_n)|\frac{x_n}{|x_n|}+\frac{r_nx}{|x_n|}|^{\alpha_n}e^{v_n(x)}
|\varphi_n|^2, \\
\slashiii{D}\varphi_n(x) &=& -V(r_nx+x_n)|\frac{x_n}{|x_n|}+\frac{r_n}{|x_n|}x|^{\alpha_n}e^{v_n(x)}\varphi_n(x), \\
\end{array}
\right.
\end{equation*} in $D_{\frac{t_nS}{r_n}} \setminus D_{\frac{\tau_nR}{r_n}}$,
and
$$ \int_{e^{-1}\leq |x|\leq 1}(|\frac{x_n}{|x_n|}+\frac{r_n}{|x_n|}x|^{2\alpha_n}e^{2v_n}+|\varphi_n|^{4})\geq \epsilon_0.
$$ By Theorem \ref{mainthm}, there are three possible cases. However, similarly to the Case I, we can rule out the first and the second possible cases. If  the third  case happens, then  for any $R>0$, there is no blow-up point in $D_R\setminus D_{\frac 1R}$ and $(v_n,\varphi_n)$ is uniformly bounded in $\overline {D_R\setminus D_{\frac 1R}}$.
Then, there is a solution $(v,\varphi)$ satisfying
\begin{equation}\label{6.3}
\left\{
\begin{array}{rcll}
-\Delta v &=& 2V^2(0)e^{2v}-V(0)e^{v}|\varphi|^2, & \quad \text{ in } \R^2\setminus\{0\} \\
\slashiii{D}\varphi &=& -V(0)e^{v}\varphi, & \quad \text{ in
}\R^2\setminus\{0\}\\
\end{array}
\right.
\end{equation}
with finite energy
$\int_{\R^2}(e^{2v}+|\varphi|^4)dx < \infty$, such that
$$\lim_{n\rightarrow \infty}  \left (  ||v_n - v ||_{C^{2}(D_R \setminus D_{\frac 1R})} +
 ||\varphi_n - \varphi||_{C^{2}(D_R\setminus D_{\frac 1R})} \right ) =0,
$$ for any $R>0$.

\
\
Next we shall use the Pohozaev identity to remove the two singularities to get another bubble.

 Firstly, since  $(u_n,\Psi_n)$ satisfies (\ref{Eq-Sn}) and \eqref{Eq-Cn} in $D_{2\delta}$, the following Pohozaev identity holds for any $\rho>0$ with $r_n\rho<  t_n$,
\begin{eqnarray*}
&&r_n\rho\int_{\partial D_{r_n\rho}(x_n)} |\frac {\partial u_n}{\partial \nu}|^2-\frac 12|\nabla u_n|^2d\sigma \\
&=&\int_{D_{r_n\rho}(x_n)}(2V^2(x)|x|^{2\alpha_n}e^{2u_n}-V(x)|x|^{\alpha_n}e^{u_n}|\Psi_n|^2)dx -r_n\rho\int_{\partial D_{r_n\rho}(x_n)}V^2(x)|x|^{2\alpha_n}e^{2u_n}d\sigma\\
& &+\frac 12\int_{\partial
D_{r_n\rho}(x_n)}\la\frac {\partial \Psi_n}{\partial \nu}, (x-x_n)\cdot\Psi_n\ra +\la
(x-x_n)\cdot\Psi_n, \frac {\partial \Psi_n}{\partial \nu}\ra d\sigma\\
& & +\int_{D_{r_n\rho}(x_n)}(e^{2u_n}(x-x_n)\cdot\nabla (V^2(x)|x|^{2\alpha_n})-e^{u_n}|\Psi_n|^2(x-x_n)\cdot \nabla
(V(x)|x|^{\alpha_n}))dx.
\end{eqnarray*}
Here we have used the fact that $|x|^{2\alpha_n}$ is smooth in $D_{r_n\rho}(x_n) \subset \R^2 \setminus \{0\}$.

Noticing again that
\begin{equation*}
\left\{
\begin{array}{rcl}
v_n(x)&=& u_n(r_nx+x_n)+\ln (r_n|x_n|^{\alpha_n}), \\
\varphi_n(x)&=& r_n^{\frac 12}\Psi_n(r_nx+x_n). \\
\end{array}
\right.
\end{equation*}
Hence, the Pohozaev constant associated with $(v_n,\varphi_n)$ (see definition \eqref{scal-2}) satisfies
\begin{eqnarray*}
& & C(v_n,\varphi_n)= C(v_n,\varphi_n,\rho)\\
& &=\rho\int_{\partial D_{\rho}} |\frac {\partial v_n}{\partial \nu}|^2-\frac 12|\nabla v_n|^2d\sigma \\
&&-\int_{D_{\rho}}(2V^2(r_nx+x_n)|\frac{x_n}{|x_n|}+\frac {r_n}{|x_n|}x|^{2\alpha_n}e^{2v_n}
-V(r_nx+x_n)|\frac{x_n}{|x_n|}+\frac {r_n}{|x_n|}x|^{\alpha_n}e^{v_n}|\varphi_n|^2)dx \\
 &&+\rho\int_{\partial D_{\rho}}V^2(r_nx+x_n)|\frac{x_n}{|x_n|}+\frac {r_n}{|x_n|}x|^{2\alpha_n}e^{2v_n}d\sigma-\frac 12\int_{\partial
D_{\rho}}\la\frac {\partial \varphi_n}{\partial \nu}, x\cdot\varphi_n\ra +\la
x\cdot\varphi_n, \frac {\partial \varphi_n}{\partial \nu}\ra d\sigma\\
& & -\int_{D_{\rho}}(e^{2v_n}x\cdot\nabla (V^2(r_nx+x_n)|\frac{x_n}{|x_n|}+\frac {r_nx}{|x_n|}|^{2\alpha_n})-e^{v_n}|\varphi_n|^2x\cdot \nabla
(V(r_nx+x_n)|\frac{x_n}{|x_n|}+\frac {r_nx}{|x_n|}|^{\alpha_n}))dx \\
&&= 0
\end{eqnarray*}
Note that $(v_n,\varphi_n)$ converges to $(v,\varphi)$ in
$C^{2}_{loc}(\R^2\setminus \{0)\})\times C^{2}_{loc}(\Gamma (\Sigma \R^2\setminus \{0\})) $ and
$|\frac{x_n}{|x_n|}+\frac {r_n}{|x_n|}x|^{\alpha_n}$ is a smooth function in $D_{\delta}$ for $\delta>0$ small enough. Therefore, we have
\begin{eqnarray*}
0 &=& \lim_{\rho \rightarrow 0} \lim_{n\rightarrow \infty}C(v_n,\varphi_n, \rho)\\
 &=& \lim_{\rho \rightarrow 0}C(v,\varphi,\rho)\\
 &&- \lim_{\delta\rightarrow 0}\lim_{n\rightarrow \infty}\int_{D_\delta}(2V^2(r_nx+x_n)|\frac{x_n}{|x_n|}+\frac {r_nx}{|x_n|}|^{2\alpha_n}e^{2v_n}
-V(r_nx+x_n)|\frac{x_n}{|x_n|}+\frac {r_nx}{|x_n|}|^{\alpha_n}e^{v_n}|\varphi_n|^2)dx\\
&=& C(v,\varphi)-\beta.
\end{eqnarray*}
Here $$\beta=\lim_{\delta \rightarrow 0}\lim_{n\rightarrow \infty}\int_{D_\delta}(2V^2(r_nx+x_n)|\frac{x_n}{|x_n|}+\frac {r_n}{|x_n|}x|^{2\alpha_n}e^{2v_n}
-V(r_nx+x_n)|\frac{x_n}{|x_n|}+\frac {r_n}{|x_n|}x|^{\alpha_n}e^{v_n}|\varphi_n|^2)dx,$$
and $C(v,\varphi)$ is the Pohozaev constant with respect to \eqref{6.3}, i.e.
\begin{eqnarray*}
C(v,\varphi)&=& \rho\int_{\partial D_\rho} |\frac {\partial v}{\partial \nu}|^2-\frac 12|\nabla v|^2d\sigma \\
&& -\int_{D_\rho}(2V^2(0)e^{2v}-V(0)e^{v}|\varphi|^2)dx \\
&& +\rho\int_{\partial D_\rho}V^2(0)e^{2v}d\sigma-\frac 12\int_{\partial
D_\rho}\la\frac {\partial \varphi}{\partial \nu}, x\cdot\varphi\ra +\la
x\cdot\varphi, \frac {\partial \varphi}{\partial \nu}\ra d\sigma.
\end{eqnarray*}

\
\

On the other hand, since  $(v_n,\varphi_n)$ converges to $(v,\varphi)$ in $C^{2}_{loc}(\R^2\setminus \{0\})\times C^{2}_{loc}(\Gamma (\Sigma \R^2\setminus \{0\})) $, we have
\begin{eqnarray*}
&&¡¡2V^2(r_nx+x_n)|\frac{x_n}{|x_n|}+\frac {r_n}{|x_n|}x|^{2\alpha_n}e^{2v_n}
-V(r_nx+x_n)|\frac{x_n}{|x_n|}+\frac {r_n}{|x_n|}x|^{\alpha_n}e^{v_n}|\varphi_n|^2\\
&&¡¡\rightarrow \nu=2V^2(0)e^{2v}-V(0)e^{v}|\varphi|^2+\beta\delta_{p=0}
\end{eqnarray*}  weakly in the sense of measures in $B_R$ for any sufficient small $R>0$.
Then, applying similar arguments as in {\bf Case I}, we can show that
$$
v(x)=-\frac{\beta}{2\pi}\log|x|+w(x)+h(x),
$$
with $w(x)$ being a bounded term and $h(x)$ being a regular term and furthermore we have
$$
C(v,\varphi)=\frac {\beta^2}{4\pi}.
$$
Hence there holds
$$
\frac {\beta^2}{4\pi}= \beta.
$$
Since $\int_{B_R}e^{2v}dx<\infty$, we have $\beta\leq 2\pi$.
Therefore we deduce that $C(v,\varphi)=0$, $\beta=0$ and hence the singularities at $0$ and $\infty$ of \eqref{6.3} can be removed. Then we get another bubble on $\S^2$. Thus we get a contradiction and complete the proof of \eqref{6.1}.

\
\

Next, similarly to {\bf Case I }, we can prove the following:

\
\

\noindent {\bf Claim II.2} :\ We can separate $A_{\delta, R, n}(x_n)$ into
finitely many parts
$$
A_{\delta, R, n}(x_n)=\bigcup_{k=1}^{N_k}A_k$$ such that on each part
\begin{equation}\label{e4}
\int_{A_k}|x|^{2\alpha_n}e^{2u_n}dx\leq \frac {1}{4\Lambda^2}, \quad k=1,2,\cdots, N_k,
\end{equation}
where  $N_k\leq N_0$ with $N_0$ being an uniform integer for all
$n$ large enough, $A_k=D_{r^{k-1}}(x_n)\setminus D_{r^k}(x_n)$, $r^0=t_n S, r^{N_k}=\tau_nR
$, $r^k<r^{k-1}$ for $k=1,2,\cdots, N_k$, and $\Lambda>0$ is the constant as in Lemma \ref{main-lamm}.

\

Then, we can use {\bf Claim II.1} and {\bf Claim II.2} to show (\ref{e3-xn}). This finishes the proof of the theorem in the second case.
\hfill{$\square$}

\
\

\section{Blow-up Behavior}\label{blow}
With  the energy identity for spinors in place, we can now rule out the possibility that $u_n$ is uniformly
bounded in $L^{\infty}_{loc}(B_r\setminus\Sigma_1)$ in Theorem \ref{mainthm} and hence the result can be improved.

\

\noindent{\bf Proof of Theorem \ref{mainthm1}:} We shall prove this by contradiction. Assume that the conclusion of the theorem is false.
Then by Theorem \ref{mainthm}, $u_n$ is uniformly bounded in $L^{\infty}$ on any compact subset of $B_r(0)\backslash\Sigma_1$.
Since $(u_n,\Psi_n)$ is a sequence of solutions to \eqref{Eq-Sn} with uniformly bounded energy (\ref{Eq-Cn}), by classical elliptic estimates for
both the Laplacian $\Delta $ and the Dirac operator $\slashiii{D}$, we know that $(u_n, \Psi_n)$
converges in $C^2$ on any compact subset of $B_r(0) \setminus \Sigma_1$ to some limit solution $(u,\Psi)$ of \eqref{Eq-L} with bounded energy $\int_{B_r(0)}(|x|^{2\alpha}e^{2u}+|\Psi|^4) < +\infty.$

Since the blow-up set $\Sigma_1$ is not empty, we can take a point $p \in \Sigma_1$. Choose a small $\delta_0>0$ such that
$p$ is the only point of $\Sigma _1$ in $\overline{B}_{2\delta_0} (p)\subset B_r(0)$. Without loss of generality, we assume that $p=0$. The case of $p\neq 0$ can be handled in an analogous way.

We shall first show that the limit $(u,\Psi)$ is $C^2$ at the isolated singularity $p=0$. In fact, since $(u_n,\Psi_n)$ satisfies the
Pohozaev identity on $D_{\rho}$ for $0<\rho<\delta_0$, the Pohozaev constant $C(u_n, \Psi_n)=C(u_n, \Psi_n,\rho)$ satisfies
\begin{eqnarray*}
0 = C(u_n, \Psi_n) &=& C(u_n, \Psi_n,\rho)\\
&=& \rho\int_{\partial D_\rho} |\frac {\partial u_n}{\partial \nu}|^2-\frac 12|\nabla u_n|^2d\sigma \\
& & -(1+\alpha_n)\int_{D_\rho}(2V^2(x)|x|^{2\alpha_n}e^{2u_n}-V(x)|x|^{\alpha_n}e^{u_n}|\Psi_n|^2)dx \\
 & & +\rho\int_{\partial D_\rho}V^2(x)|x|^{2\alpha_n}e^{2u_n}d\sigma-\frac 12\int_{\partial
D_\rho}\la\frac {\partial \Psi_n}{\partial \nu}, x\cdot\Psi_n\ra +\la
x\cdot\Psi_n, \frac {\partial \Psi_n}{\partial \nu}\ra d\sigma\\
& & -\int_{D_{\rho}}(|x|^{2\alpha_n}e^{2u_n}x\cdot\nabla (V^2(x))-|x|^{\alpha_n}e^{u_n}|\Psi_n|^2x\cdot \nabla
V(x))dx.
\end{eqnarray*}
Since $(u_n,\Psi_n)$ converges to $(u,\Psi)$ in $C^2$ on any compact subset of $\overline{B}_{2\delta_0} \setminus \{0\}$, we have
\begin{eqnarray*}
0 &=& \lim_{\rho\rightarrow 0}\lim_{n\rightarrow \infty}C(u_n,\Psi_n, \rho)\\
 &=&\lim_{\rho\rightarrow 0}C(u,\Psi, \rho) -(1+\alpha)\lim_{\delta\rightarrow 0}\lim_{n\rightarrow \infty}\int_{D_\delta}(2V^2(x)|x|^{2\alpha_n}e^{2u_n}-V(x)|x|^{\alpha_n}e^{u_n}|\Psi_n|^2)dx\\
&=&C(u,\Psi)-(1+\alpha)\beta
\end{eqnarray*}
where
\begin{eqnarray*}
C(u,\Psi) &=& C(u, \Psi,\rho)\\
&=& \rho\int_{\partial D_\rho} |\frac {\partial u}{\partial \nu}|^2-\frac 12|\nabla u|^2d\sigma \\
& & -(1+\alpha)\int_{D_\rho}(2V^2(x)|x|^{2\alpha}e^{2u}-V(x)|x|^{\alpha}e^{u}|\Psi|^2)dx \\
 & & +\rho\int_{\partial D_\rho}V^2(x)|x|^{2\alpha}e^{2u}d \sigma-\frac 12\int_{\partial
D_\rho}\la\frac {\partial \Psi}{\partial \nu}, x\cdot\Psi\ra +\la
x\cdot\Psi, \frac {\partial \Psi}{\partial \nu}\ra d\sigma\\
& & -\int_{D_{\rho}}(|x|^{2\alpha}e^{2u}x\cdot\nabla (V^2(x))-|x|^{\alpha}e^{u}|\Psi|^2x\cdot \nabla
V(x))dx,
\end{eqnarray*}
and
$$\beta=\lim_{\delta\rightarrow 0}\lim_{n\rightarrow \infty}\int_{D_\delta}(2V^2(x)|x|^{2\alpha_n}e^{2u_n}-V(x)|x|^{\alpha_n}e^{u_n}|\Psi_n|^2)dx.$$
Moreover, we can also assume  that
\begin{eqnarray*}
2V^2(x)|x|^{2\alpha_n}e^{2u_n}-V(x)|x|^{\alpha_n}e^{u_n}|\Psi_n|^2 \rightarrow \nu =2V^2(x)|x|^{2\alpha}e^{2u}-V(x)|x|^{\alpha}e^{u}|\Psi|^2+\beta\delta_{p=0}
\end{eqnarray*}in the sense of distributions in $B_R$ for any small $R>0$. Then, applying similar arguments as in the proof of the local singularity removability in Claim I.1, Theorem \ref{engy-indt}, we can show that  $C(u,\Psi)=0$, $\beta = 0$ and hence $(u, \Psi)$ is a $C^2$ solution of \eqref{Eq-L} on $B_{2\delta_0}$ with bounded energy
$$\int_{B_{2\delta_0}}(|x|^{2\alpha} e^{2u}+|\Psi|^4)< +\infty.$$

\

Now we can choose some small $\delta_1 \in (0, \delta_0)$ such that for any $\delta \in (0, \delta_1)$,
\begin{equation} \label{less1/10}
\int_{B_{\delta}}(2V^2(x)|x|^{2\alpha}e^{2u}-V(x)|x|^{\alpha}e^{u}|\Psi|^2)dx < {\rm min} \{ \frac{1+\a}{10},  \frac{1}{10}\}.
\end{equation}

Next, as  in the proof of Theorem \ref{engy-indt}, we rescale $(u_n,\Psi_n)$ near $p=0$.
Choose $x_n\in B_{\delta_1}$ with $u_n(x_n)=\max_{\bar{B}_{\delta_1}}u_n(x)$. Then we have $x_n\rightarrow p$ and
$u_n(x_n)\rightarrow +\infty$. Let $\lambda_n=e^{\frac{-u_n(x_n)}{\alpha_n+1}}\rightarrow 0$ and denote $t_n=\max\{\lambda_n, |x_n|\}\rightarrow 0$.
We distinguish the following two cases:

\
\

\noindent{\bf Case I}:
$\frac{t_n}{\lambda_n}=O(1), \text{ as } n\rightarrow +\infty.
$

\
\

In this case, the rescaling functions are
\begin{equation*}
\left\{
\begin{array}{rcl}
\widetilde{u}_n(x)&=&u_n(t_nx)+(\alpha_n+1)\ln {t_n}\\
\widetilde{\Psi}_n(x)&=&t_n^{\frac
12}\Psi_n(t_nx)
\end{array}
\right.
\end{equation*}
for any $x\in \overline{D}_{\frac{\delta_1}{2t_n}}$. And by passing to a subsequence,
 $(\widetilde{u}_n,\widetilde{\Psi}_n)$ converges in $C^{2}_{loc}(\R^2)$ to some $(\widetilde
u,\widetilde \Psi)$ satisfying
\begin{equation*}
\left\{
\begin{array}{rcl}
-\Delta \widetilde{u} &=& 2V^2(0)|x|^{2\alpha}e^{2\widetilde{u}}-V(0)|x|^{\alpha}e^{\widetilde u}|
\Psi|^2, \\
\slashiii{D}\widetilde{\Psi} &=&
-V(0)|x|^{\alpha}e^{\widetilde{u}}\widetilde{\Psi},  \end{array} \text { in } \R^2
\right.
\end{equation*}
with
\begin{equation}
\label{7.1}\int_{\R^2}(2V^2(0)|x|^{2\alpha}e^{2\widetilde{u}}-V(0)|x|^{\alpha}e^{\widetilde {u}}|\widetilde{\Psi}|^2)dx=4\pi(1+\alpha).
\end{equation}
Then for $\delta\in(0, \delta_1)$ small enough, $R>0$ large enough  and $n$ large enough, we have
\begin{eqnarray}
& & \int_{B_{\delta}}(2V^2(x)|x|^{2\alpha_n}e^{2u_n}-V(x)|x|^{\alpha_n}e^{u_n}|\Psi_n|^2)dx \nonumber\\
&=&\int_{B_{t_n R}}(2V^2(x)|x|^{2\alpha_n}e^{2u_n}-V(x)|x|^{\alpha_n}e^{u_n}|\Psi_n|^2)dx\nonumber\\
& & +\int_{B_\delta\backslash B_{t_nR}}(2V^2(x)|x|^{2\alpha_n}e^{2u_n}-V(x)|x|^{\alpha_n}e^{u_n}|\Psi_n|^2)dx\nonumber\\
&\geq & \int_{B_{R}}(2V^2(t_nx)|x|^{2\alpha_n}e^{2\widetilde{u}_n}-V(t_nx)|x|^{\alpha_n}e^{\widetilde{u}_n}|\widetilde{\Psi}_n|^2)
 -\int_{B_\delta\backslash B_{t_nR}}V(x)|x|^{\alpha_n}e^{u_n}|\Psi_n|^2 \nonumber\\
& \geq &  4\pi(1+\alpha) - \frac{1+\a}{10}.  \label{almost2pi}
\end{eqnarray}
Here in the last step, we have used \eqref{7.1} and the fact from Theorem \ref{engy-indt} that the neck energy of the spinor field $\Psi_n$ is converging to zero. We remark that in the above estimate, if there are multiple bubbles then we need to decompose
$B_\delta\backslash B_{t_nR}$ further into bubble domains and neck domains and then apply the no neck energy result in Theorem \ref{engy-indt}
to each of these neck domains.

On the other hand, we fix some $\delta\in(0, \delta_1)$ small such that \eqref{almost2pi} holds and then let $n\rightarrow \infty$
to conclude that
\begin{eqnarray*}
4\pi(1+\alpha) - \frac{1+\a}{10} &\leq & \int_{B_{\delta}}(2V^2(x)|x|^{2\alpha_n}e^{2u_n}-V(x)|x|^{\alpha_n}e^{u_n}|\Psi_n|^2)dx\\
&  =& -\int_{B_{\delta}}\Delta u_n= -\int_{\partial B_{\delta}}\frac{\partial u_n}{\partial n}\\
& \rightarrow &   -\int_{\partial B_{\delta}}\frac{\partial u}{\partial n} =  -\int_{B_{\delta}}\Delta u\\
& =& \int_{B_{\delta}}(2V^2(x)|x|^{2\alpha}e^{2u}-V(x)|x|^{\alpha}e^{u}|\Psi|^2)dx < \frac{1+\a}{10}
\end{eqnarray*}
Here in the last step, we have used \eqref{less1/10}. Thus we get a contradiction and finish the proof of the Theorem in this case.

\
\

\noindent{\bf Case II}:  $\frac{t_n}{\lambda_n}\rightarrow +\infty$, as $n\rightarrow \infty.$

\
\

In this case, we should  rescale twice to get the bubble.
First, since $t_n=|x_n|$, we define the rescaling functions

\begin{equation*}
\left\{
\begin{array}{rcl}
\bar{u}_n(x)&=&u_n(t_nx)+(\alpha_n+1)\ln {t_n}\\
\bar{\Psi}_n(x)&=&t_n^{\frac
12}\Psi_n(t_nx)
\end{array}
\right.
\end{equation*}
for any $x\in \overline{D}_{\frac{\delta}{2t_n}}$. Set  $y_n:=\frac{x_n}{t_n}$. Noticing that $\bar{u}_n(y_n)\rightarrow +\infty$, we set that $\delta_n=e^{-\bar{u}_n(y_n)}$ and define the rescaling function
\begin{equation*}
\left\{
\begin{array}{rcl}
\widetilde{u}_n(x)&=&\bar{u}_n(\delta_nx+y_n)+\ln {\delta_n}\\
\widetilde{\Psi}_n(x)&=&\delta_n^{\frac 12}\bar{\Psi}_n(\delta_nx+y_n)
\end{array}
\right.
\end{equation*}
for any $\delta_nx+y_n\in \overline{D}_{\frac{\delta}{2t_n}}$. Without loss of generality, we assume that $y_0=\lim_{n\rightarrow \infty}\frac{x_n}{t_n}$.
Then  by also passing to a subsequence,
$(\widetilde{u}_n,\widetilde{\Psi}_n)$ converges in $C^{2}_{loc}(\R^2)$ to some $(\widetilde
u,\widetilde \Psi)$ satisfying
\begin{equation*}
\left\{
\begin{array}{rcl}
-\Delta \widetilde{u} &=& 2V^2(0)e^{2\widetilde{u}}-V(0)e^{\widetilde u}|
\Psi|^2, \\
\slashiii{D}\widetilde{\Psi} &=&
-V(0)e^{\widetilde{u}}\widetilde{\Psi},  \end{array} \text { in } \R^2
\right.
\end{equation*}
with
\begin{equation}\label{7.2} \int_{\R^2}(2V^2(0)e^{2\widetilde{u}}-V(0)e^{\widetilde{u}}|\widetilde{\Psi}|^2)dx=4\pi.
\end{equation}
Now fixing $\delta\in(0, \delta_1)$ small enough, $S, R>0$ large enough  and $n$ large enough,
by using \eqref{7.2} and the fact that the neck energy of the spinor field $\Psi_n$
is converging to zero, we have
\begin{eqnarray*}
 && \int_{B_{\delta}}(2V^2(x)|x|^{2\alpha_n}e^{2u_n}-V(x)|x|^{\alpha_n}e^{u_n}|\Psi_n|^2)dx \nonumber\\
&=&\int_{B_{\frac{\delta}{t_n}}}(2V^2(t_nx)|x|^{2\alpha_n}e^{2\bar{u}_n}-V(t_nx)|x|^{\alpha_n}e^{\bar{u}_n}|\bar{\Psi}_n|^2)dx \nonumber\\
&=&\int_{B_{\frac{\delta}{t_n}}\setminus B_S(y_n)}(2V^2(t_nx)|x|^{2\alpha_n}e^{2\bar{u}_n}-V(t_nx)|x|^{\alpha_n}e^{\bar{u}_n}|\bar{\Psi}_n|^2)dx \nonumber\\
&&+\int_{ B_S(y_n)\setminus B_{\frac{\tau_n}{t_n}R}(y_n)}(2V^2(t_nx)|x|^{2\alpha_n}e^{2\bar{u}_n}-V(t_nx)|x|^{\alpha_n}e^{\bar{u}_n}|\bar{\Psi}_n|^2)dx \nonumber\\
&&+\int_{B_{\frac{\tau_n}{t_n}R}(y_n)}(2V^2(t_nx)|x|^{2\alpha_n}e^{2\bar{u}_n}-V(t_nx)|x|^{\alpha_n}e^{\bar{u}_n}|\bar{\Psi}_n|^2)dx \nonumber\\
&\geq & \int_{B_{R}}(2V^2(x_n+\tau_nx)|\frac{x_n}{|x_n|}+\frac{\tau_n}{|x_n|}x|^{2\alpha_n}e^{2\widetilde{u}_n(x)}
-V(x_n+\tau_nx)|\frac{x_n}{|x_n|}+\frac{\tau_n}{|x_n|}x|^{\alpha_n}
e^{\widetilde{u}_n(x)}|\widetilde{\Psi}_n|^2)dx \nonumber\\
& & -\int_{B_{t_nS}\backslash B_{\tau_nR}(x_n)}V(x)|x|^{\alpha_n}e^{u_n}|\Psi_n|^2 - \int_{B_{\frac{\delta}{t_n}}\setminus B_S(y_n)}V(t_nx)|x|^{\alpha_n}e^{\bar{u}_n}|\bar{\Psi}_n|^2 \nonumber\\
& \geq &  4\pi - \frac{1}{10}.
\end{eqnarray*}
Then, applying similar arguments as in {\bf Case I} we get a contradiction. Thus we finish the proof of the Theorem.
\hfill{$\square$}

\
\

\section{Blow-up Value}\label{value}

In this section, we shall further investigate the blow-up behavior of a sequence of solutions of (\ref{Eq-Sn}) and (\ref{Eq-Cn}).
Let $m(p)$ be the blow-up value at a blow-up point $p \in \Sigma_1$ defined as in \eqref{mp}. It is clear from the result in Theorem \ref{mainthm1} that $m(p)\geq 4\pi$. Now we shall determine the precise value of $m(p)$ under a boundary condition.

\

\noindent{\bf Proof of Theorem \ref{BV}:} Without loss of generality, we assume $p=0$. The case of $p\neq 0$ can be handled analogously.
 It follows from the boundary condition in
(\ref{bc}) that $0\leq u_n-\min_{\partial B_{\rho_0}(p)}u_n\leq C$ on
$\partial B_{\rho_0}(p)$. Define $w_n$ as the unique solution of the following Dirichlet problem
\begin{equation*}
\left\{
\begin{array}{lcl}
-\Delta w_n=0, &\text{in } B_{\rho_0}(p),\\
w_n=u_n-\min_{\partial B_{\rho_0}}u_n, &\text{on }\partial B_{\rho_0}(p).
\end{array}
\right.
\end{equation*}
By the maximum principle,  $w_n$ is uniformly bounded
in $B_{\rho_0}(p)$ and consequently $w_n$ is $C^2$ in $B_{\rho_0}(p)$. Furthermore, the function $v_n=u_n-\min_{\partial
B_{\rho_0}(p)}u_n-w_n$ solves the Dirichlet problem
\begin{equation*}
\left\{
\begin{array}{lcl}
-\Delta v_n=2V^2(x)|x|^{2\alpha_n}e^{2u_n}-V(x)|x|^{\alpha_n}e^{u_n}|\Psi_n|^2, &\text{in } B_{\rho_0}(p),\\
v_n=0, &\text{on }\partial B_{\rho_0}(p),
\end{array}
\right.
\end{equation*}
with the energy condition
$$\int_{B_{\rho_0}(p)}(2V^2(x)|x|^{2\alpha_n}e^{2u_n}-V(x)|x|^{\alpha_n}e^{u_n}|\Psi_n|^2)dx\leq C. $$
By  Green's representation formula, we have
\[v_n(x)=\frac 1{2\pi}\int_{B_{\rho_0}(p)}\log\frac
1{|x-y|}(2V^2(y)|y|^{2\alpha_n}e^{2u_n}-V(y)|y|^{\alpha_n}e^{u_n}|\Psi_n|^2)dy+R_n(x)
\]
where $R_n(x) \in C^1 (B_{\rho_0}(p))$ is a regular term. Since $p=0$ is the only blow-up point in $B_{\rho_0}(p)$, from Theorem \ref{mainthm1}, we know
\begin{equation}\label{8.2}
v_n(x)\rightarrow \frac{m(p)}{2\pi}\ln\frac{1}{|x|}+R(x), \text{ in } C^{1}_{loc}(B_{\rho_0}(p)\setminus\{0\})
\end{equation}
for $R(x)\in C^1(B_{\rho_0}(p))$.
On the other hand, we observe that $(v_n,\Psi_n)$ satisfies
\begin{equation*}
\left\{
\begin{array}{lcl}
-\Delta v_n &=& 2K_n^2(x)|x|^{2\alpha_n}e^{2v_n}-K_n(x)|x|^{\alpha_n}e^{v_n}|\Psi_n|^2, \\
\slashiii{D}{\Psi_n} &=&-K_n(x)e^{v_n}{\Psi_n},
\end{array}\quad \text{in } B_{\rho_0}(p)
\right.
\end{equation*}
where $K_n=V(x)e^{\min_{\partial B_{\rho_0}(p)}u_n+w_n}$. Noticing that $p=0$, the Pohozaev identity of $(v_n,\Psi_n)$ in $B_{\rho}(p)$ for $0<\rho<\rho_0$ is
\begin{eqnarray}\label{8.3}
&&\rho\int_{\partial B_\rho(0)} |\frac {\partial v_n}{\partial \nu}|^2-\frac 12|\nabla v_n|^2d\sigma \nonumber\\
&=&(1+\alpha_n                     )\int_{B_\rho(0)}(2K_n^2(x)|x|^{2\alpha_n}e^{2v_n}-K_n(x)|x|^{\alpha_n}e^{v_n}|\Psi_n|^2)dx \nonumber\\
& & -\rho\int_{\partial B_\rho(0)}K_n^2(x)|x|^{2\alpha_n}e^{2v_n}d\sigma+\frac 12\int_{\partial
B_\rho(0)}\la\frac {\partial \Psi_n}{\partial \nu}, x\cdot\Psi_n\ra +\la
x\cdot\Psi_n, \frac {\partial \Psi_n}{\partial \nu}\ra d\sigma\nonumber\\
& & +\int_{B_\rho(0)}(|x|^{2\alpha_n}e^{2v_n}x\cdot\nabla (K_n^2(x))-|x|^{\alpha_n}e^{v_n}|\Psi_n|^2x\cdot \nabla
K_n(x))dx.
\end{eqnarray}

By (\ref{8.2}), we have
$$
\lim_{\rho\rightarrow 0}\lim_{n\rightarrow \infty}\rho\int_{\partial B_\rho(0)}
|\frac {\partial v_n}{\partial \nu}|^2-\frac 12|\nabla
v_n|^2d\sigma=\lim_{\rho\rightarrow 0}\rho\int_{\partial B_\rho(0)}\frac 12 |\frac
{\partial (\frac{m(p)}{2\pi}\ln\frac{1}{|x|})}{\partial \nu}|^2d\sigma=\frac
1{4\pi}m^2(p).
$$
Since $u_n\rightarrow -\infty$ uniformly on $\partial B_{\rho}(0)$, we also have
$$
\lim_{\rho\rightarrow 0}\lim_{n\rightarrow \infty} \rho\int_{\partial B_\rho(0)}K_n^2(x)|x|^{2\alpha_n}e^{2v_n}d\sigma =\lim_{\rho\rightarrow 0}\lim_{n\rightarrow \infty}
\rho\int_{\partial B_\rho(0)}V^2(x)|x|^{2\alpha_n}e^{2u_n}d\sigma=0.
$$
Noticing that  $\int_{B_{\rho_0}(0)}(|x|^{2\alpha_n}e^{2u_n}+\Psi_n|^4)dx\leq C$, we can obtain that
$$
\lim_{\rho\rightarrow 0}\lim_{n\rightarrow \infty}\int_{B_\rho(0)}(|x|^{2\alpha_n}e^{2v_n}x\cdot\nabla (K_n^2(x))-|x|^{\alpha_n}e^{v_n}|\Psi_n|^2x\cdot \nabla
K_n(x))dx=0.
$$
Since $u_n\rightarrow -\infty$ uniformly in $B_{2\rho}(0)\backslash B_{\frac \rho 4}(0)$, and  $|\Psi_n|$ is uniformly bounded in
$B_{2\rho}(0)\backslash B_{\frac \rho4}(0)$ for any $\rho>0$, we know
$$\slashiii{D}\Psi=0,  ~~~ \text{ in } B_{\rho_0}\setminus\{0\}.$$
Since the local singularity of a harmonic spinor with finite energy is removable, we have
$$\slashiii{D}\Psi=0,  ~~~ \text{ in } B_{\rho_0}.$$
It follows that $\Psi$ is smooth in $B_{\rho_0}$.
Therefore we obtain that
$$
\lim_{\rho\rightarrow 0}\lim_{n\rightarrow \infty}\int_{\partial
B_\rho(0)}|\Psi_n||x\cdot \nabla \Psi_n|d\sigma=0.
$$
Let $n\rightarrow \infty$ and then $\rho\rightarrow 0$  in (\ref{8.3}), we
get that $$ \frac{1}{4\pi}m^2(p)=(1+\alpha)m(p).$$ It follows that
$m(p)=4\pi(1+\alpha)$. Thus we finish the proof of Theorem \ref{BV}.
\hfill{$\square$}

\
\section{The global super-Liouville system on a singular Riemann surface}\label{global}

In this section, we study the blow-up behavior of a sequence of solutions of the global super-Liouville system
on a singular Riemann surface and  prove Theorem \ref{thmsin} and Theorem \ref{global-blowup}.

\

\noindent{\bf Proof of Theorem \ref{thmsin}:}
Since $g=e^{2\phi}g_0$ with $g_0$ being smooth, then by the well known properties of $\phi$ (see e.g. \cite{T1} or \cite{BDM}, p. 5639), we know that $(u_n,\psi_n)$ satisfies
\begin{equation*}
\left\{
\begin{array}{rcl}
-\Delta_{g_0} (u_n+\phi) &=& \ds\vs 2e^{2(u_n+\phi)}-e^{u_n+\phi}\left\langle e^{\frac {\phi} 2}\psi_n ,e^{\frac {\phi} 2}\psi_n
\right\rangle -K_{g_0}  -   \sum_{j=1}^m  2\pi\alpha_j \delta_{q_j}   \qquad
\\
\slashiii{D}_{g_0}(e^{\frac {\phi} 2}\psi_n) &=&\ds  -e^{u_n+\phi}(e^{\frac {\phi} 2}\psi_n)
\end{array} \text {in } M.
\right.
\end{equation*}
with the energy conditions:
$$
\int_{M}e^{2(u_n+\phi)}dg_0<C,~~~~~~\int_{M}|e^{\frac {\phi} {2} }\psi_n|^4dg_0<C.
$$
If we define the blow-up set of $u_n + \phi$ as
$$
\Sigma '_1=\left\{ x\in M,\text{ there is a sequence
}y_n\rightarrow x\text{ such that }(u_n+\phi)(y_n)\rightarrow +\infty
\right\}, $$
then by Remark \ref{rem3.4}, we have  $\Sigma_1=\Sigma'_1$. By the blow-up results of the local sytem, it follows that one of the following alternatives holds:
\begin{enumerate}
\item[i)]  $u_n$ is bounded in $L^{\infty}(M).$

\item[ii)]  $u_n$ $\rightarrow -\infty $ uniformly on $M$.

\item[iii)]  $\Sigma _1$ is finite, nonempty and
\begin{equation*}
u_n\rightarrow -\infty \text{ uniformly on compact subsets of
}M\backslash \Sigma _1.
\end{equation*}
Furthermore,
\begin{equation*}
2e^{2(u_n+\phi)}-e^{u_n+\phi}|e^{\frac{\phi}{2}}\psi_n|^2\rightharpoonup\sum_{p_i\in \Sigma_1}m(p_i)\delta_{p_i},
\end{equation*} in the sense of distributions.
\end{enumerate}

\
\

Now let $p=\frac q{q-1}>2$. We have
\begin{eqnarray*}
& & ||\nabla (u_n+\phi)||_{L^q(M,g_0)}\\
& \leq & \sup\{|\int_M\nabla (u_n+\phi)\nabla \varphi
dg_0||\varphi\in W^{1,p}(M,g_0),\int_{M}\varphi
dg_0=0,||\varphi||_{W^{1,p}(M,g_0)}=1\}.
\end{eqnarray*}
By the Sobolev embedding theorem, we get
$$
||\varphi ||_{L^{\infty}(M,g_0)\leq C}.
$$
It is clear that
\bee
|\int_{M}\nabla (u_n+\phi)\nabla\varphi dg_0| &=& |\int_M\Delta_{g_0} (u_n+\phi)\varphi
dg_0|  \nn \\
&\leq& \int_{M}(2e^{2(u_n+\phi)}+e^{u_n+\phi}|e^{\frac {\phi} {2}} \psi_n|^2 + |K_{g_0}|)|\varphi|dg_0
 + \sum_{j=1}^m  |\int_M 2\pi\alpha_j \delta_{q_j}   \varphi dg_0|\leq C.  \nn
\eee
Therefore, $u_n+\phi-\frac{1}{|M|}\int_{M}(u_n+\phi)dg_0$ is uniformly bounded in $W^{1,q}(M,g_0)$.

\
\

Next, we define the Green function $G$ by
\begin{equation*}
\left \{\begin{array}{l} - \Delta_{g_0}
G=\sum_{p\in\Sigma_1}m(p)\delta_p-K_{g_0}-\sum_{j=1}^m 2\pi\alpha_j \delta_{q_j},\\
\int_{M}Gdg_0=0.
\end{array}\right.
\end{equation*}
Then $G$ satisies \eqref{greenfct}. We have for any $\varphi\in C^{\infty}(M)$
\begin{eqnarray*}
& & \int_{M}\nabla (u_n+\phi-G)\nabla
\varphi dg_0 = -\int_{M}\Delta_{g_0}(u_n+\phi-G)\varphi dg_0\\
&=& \int_{M}( 2e^{2(u_n+\phi)}-e^{u_n+\phi}\left\langle e^{\frac {\phi} 2}\psi_n ,e^{\frac {\phi} 2}\psi_n
\right\rangle-\sum_{p\in
\Sigma_1}m(p)\delta_p)\varphi dg_0\rightarrow 0,\quad \text { as }
n\rightarrow \infty.
\end{eqnarray*}
Combining this with the fact that $u_n+\phi-\frac{1}{|M|}\int_{M}(u_n+\phi)dg_0$ is uniformly bounded in $W^{1,q}(M,g_0)$, we get the conclusion of the lemma.
\hfill{$\square$}

\
\

\noindent{\bf Proof of Theorm \ref{global-blowup}:} The result follows from Theorem \ref{thmsin} and the Gauss-Bonnet formula.\hfill{$\square$}

\

\section{The local super-Liouville equations with two coefficient functions}\label{two}

 In this section, we discuss  the following local super-Liouville type equations with two different coefficient functions:
\begin{equation}\label{equvc}
\left\{
\begin{array}{rcl}
-\Delta u(x) &=& 2V^2(x)|x|^{2\alpha}e^{2u(x)}-W(x)|x|^{\alpha}e^{u(x)}|\Psi|^2  \quad\\
\slashiii{D}\Psi &=& -W(x)|x|^{\alpha}e^{u(x)}\Psi
\end{array} \text { in } B_{r}(0),
\right.
\end{equation}
and with the energy condition
\begin{equation}\label{cequvc}
\int_{B_r(0)}|x|^{2\alpha}e^{2u}+|\Psi|^4dx<+\infty,
\end{equation}
where $\a > -1$ and $V(x), W(x) \in W^{1,\infty}(\overline{B_r(0)})$ satisfying $0< a\leq V(x), W(x) \leq b < + \infty$.
In analogy to the case considered in Section \ref{local}, we can define the notion of weak solutions
$(u,\Psi)\in W^{1,2}(B_r(0)) \times W^{1,\frac{4}{3}} (\Gamma(\Sigma B_r(0)))$ of (\ref{equvc}) and  (\ref{cequvc}) and
show that any such weak solution $(u,\Psi)$ is {\it regular} in the sense that $(u,\Psi)\in W^{1,p}(B_r(0)) \times W^{1,q} (\Gamma(\Sigma B_r(0)))$
 for some $p>2$ and some $q>2$ and $(u,\Psi)$ is  $C^{2}_{loc}\times C^{2}_{loc}$ in $B_r(0)\setminus \{0\} $.

\

Firstly, it is easy to check that the following Pohozaev type identity holds:
\begin{eqnarray*}
&&R\int_{\partial B_R(0)} |\frac {\partial u}{\partial \nu}|^2-\frac 12|\nabla u|^2d\sigma \nn\\
&=&(1+\alpha)\int_{B_R(0)}(2V^2(x)|x|^{2\alpha}e^{2u}-W(x)|x|^{\alpha}e^u|\Psi|^2)dx \nn\\
& & -R\int_{\partial B_R(0)}V^2(x)|x|^{2\alpha}e^{2u}d\sigma+\frac 12\int_{\partial
B_R(0)}(\la\frac {\partial \Psi}{\partial \nu}, x\cdot\Psi\ra +\la
x\cdot\Psi, \frac {\partial \Psi}{\partial \nu}\ra ) d\sigma   \nn\\
& & +\int_{B_R(0)}(|x|^{2\alpha}e^{2u}x\cdot\nabla (V^2(x))-|x|^{\alpha}e^u|\Psi|^2x\cdot \nabla
W(x))dx
\end{eqnarray*} for any regular solution $(u,\Psi)$ of (\ref{equvc}) and  (\ref{cequvc}) on $B_r(0)$ and for any $0<R<r$.

Secondly, when $(u,\Psi)$ is a regular solution of \eqref{equvc} and \eqref{cequvc} in $B_r(0)\setminus \{0\}$, we define the Pohozaev constant associated to $(u, \Psi)$ as follows
 \begin{eqnarray*}
C(u,\Psi)&=& R\int_{\partial B_R(0)} |\frac {\partial u}{\partial \nu}|^2-\frac 12|\nabla u|^2d\sigma \\
&-&(1+\alpha)\int_{B_R(0)}(2V^2(x)|x|^{2\alpha}e^{2u}-W(x)|x|^{\alpha}e^u|\Psi|^2)dx \\
& & +R\int_{\partial B_R(0)}V^2(x)|x|^{2\alpha}e^{2u}d\sigma-\frac 12\int_{\partial
B_R(0)}\la\frac {\partial \Psi}{\partial \nu}, x\cdot\Psi\ra +\la
x\cdot\Psi, \frac {\partial \Psi}{\partial \nu}\ra d\sigma\\
& & -\int_{B_R(0)}(|x|^{2\alpha}e^{2u}x\cdot\nabla (V^2(x))-|x|^{\alpha}e^u|\Psi|^2x\cdot \nabla
W(x))dx.
\end{eqnarray*}
Then the local singularity removability as in Proposition \ref{sigu-move1} holds.

Thirdly, for a bubble, namely an entire regular solution on $\R^2$ with bounded energy, we consider the following equation:
\begin{equation*}
\left\{
\begin{array}{rcl}
-\Delta u &=& \ds\vs 2a|x|^{2\alpha} e^{2u}-b|x|^{\alpha}e^u|\Psi|^2,\\
\slashiii{D}\Psi &=&\ds  -b|x|^{\alpha} e^u\Psi,
\end{array}\qquad \text { in } \R^2.
\right.
\end{equation*}
with $\alpha >-1 $ and for two real numbers $a>0$ and $b>0$. The energy condition is
\begin{equation*}
I(u,\Psi)=\int_{\R^2}(|x|^{2\alpha}e^{2u}+|\Psi|^4)dx<\infty.
\end{equation*}
By using its corresponding Pohozaev type identy, we can prove the same results as in Proposition \ref{asy} and Theorem \ref{remove-gs}. In particular, we have
$$ d =\int_{\R^2}2a|x|^{2\alpha}e^{2u} -b|x|^{\alpha}e^u|\Psi|^2dx=4\pi(1+\alpha).$$

Finally, for a sequence of regular solutions $(u_n, \Psi_n)$ to (\ref{Eq-Sn2}) and (\ref{Eq-Cn2}),  we define the blow-up value $m(p)$ at a blow-up point $p$ as
\begin{equation*}
m(p)=\lim_{\rho\rightarrow 0}\lim_{n\rightarrow \infty}\int_{B_\rho(p)}(2V_n^2(x)|x|^{2\alpha_n}e^{2u_n}-W_n(x)|x|^{\alpha_n}e^{u_n}|\Psi_n|^2)dx,
\end{equation*}
and we can show that the blow-up behaviors for $(u_n, \Psi_n)$ as in Theorem \ref{mainthm}, Theorem \ref{engy-indt}, Theorem \ref{mainthm1}, Theorem \ref{BV}
and Theorem \ref{thmsin} hold.

\

\end{document}